\documentclass[12pt,a4paper]{article} 
\usepackage{amsmath}
\usepackage{amsthm}
\usepackage{amsfonts}
\usepackage{amssymb}
\usepackage{tikz}
\tikzset{node distance=2cm, auto}
\usepackage{bussproofs}
\author{Amir Akbar Tabatabai}
\theoremstyle{plain} 
\newtheorem{thm}{Theorem}[section]

\newtheorem{lem}[thm]{Lemma}

\newtheorem{cor}[thm]{Corollary}
\theoremstyle{definition}
\newtheorem{dfn}[thm]{Definition}
\newtheorem{exam}[thm]{Example}
\newtheorem{rem}[thm]{Remark}

\def\int{\mathrm{int}}

\def\exp{\mathrm{exp}}
\def\CHaus{\mathrm{CHaus}}

\def\Top{\mathrm{Top}}
\def\S4{\mathrm{S4}}
\def\Hom{\mathrm{Hom}}
\def\Set{\mathrm{Set}}

\begin{document}
\title{Geometric Modality and Weak Exponentials} 

\author{Amirhossein Akbar Tabatabai \footnote{The author is supported by the ERC Advanced Grant 339691 (FEALORA)}\\
Institute of Mathematics\\
Academy of Sciences of the Czech Republic\\
tabatabai@math.cas.cz}

\date{}

\maketitle

\begin{abstract}
The intuitionistic implication and hence the notion of function space in constructive disciplines is both non-geometric and impredicative. In this paper we try to solve both of these problems by first introducing weak exponential objects as a formalization for predicative function spaces and then by proposing modal spaces as a way to introduce a natural family of geometric predicative implications based on the interplay between the concepts of time and space. This combination then leads to a brand new family of modal propositional logics with predicative implications and then to topological semantics for these logics and some weak modal and sub-intuitionistic logics, as well. Finally, we will lift these notions and the corresponding relations to a higher and more structured level of modal topoi and modal type theory.
\end{abstract}

\section{Introduction}
Intuitionistic logic appears in different many branches of mathematics with many different and interesting incarnations. In geometrical world it plays the role of the language of a topological space via topological semantics and in a higher and more structured level, it becomes the internal logic of any elementary topoi. On the other hand and in the theory of computations, the intuitionistic logic shows its computational aspects as a method to describe the behavior of computations using realizability interpretations and in category theory it becomes the syntax of the very central class of Cartesian closed categories. In all of these incarnations though, we can name some unexpected behaviors that relates to the definition of the intuitionistic implication and as we believe its impredicative definition as a part of the well-known BHK interpretation. In the rest of this introduction we will explain what these behaviors are, why they are unexpected and how we can resolve the situation.\\

Let us begin our journey by the interesting case of geometry. Since Tarski's seminal work \cite{Mc}, it has been well-known that the intuitionistic logic can be interpreted, in a sound and complete way, in topological spaces. In this sense, it is meaningful to assume that in contrast to the fact that the classical logic is the logic of sets, the intuitionistic logic should be considered as \textit{the} logic of the notion of space. This assumption becomes more evident when we see the appearance of intuitionistic logic in the generalized geometrical discourses such as the topos theory, or more recently, the homotopy theory. However, the situation is not as clear as one expects. A geometrical concept is not just anything that appears in all the geometrical situations but a concept which is preserved by the morphisms of the discourse and this is exactly where some problems come to the scene. Let us explain the phenomenon more precisely using the simpler case of topological spaces as the topological models of the propositional intuitionistic logic. The definition of the topological interpretation is the following: Let $(X, \tau)$ be a topological space and $V$ an interpretation which assigns open subsets of $X$ to the atomic formulas in the language. Next, extend $V$ to the class of all formulas in the following natural way: Interpret conjunction, disjunction, $\bot$ and $\top$ as the intersection, union, the empty subset and the whole subset $X$, respectively. And interpret the implication $A \rightarrow B$ as $int(V(A)^c \cup V(B))$. From the intuition that we have explained already, it is natural to interpret $V(C)$ as the geometrical content of the sentence $C$ in the topological space $(X, \tau)$ and then expect that all continuous functions preserve this geometrical content. More precisely, we expect that if $f:(X, \tau) \to (Y, \eta)$ is continuous, then $f^{-1}$ sends the geometrical content of any formula $C$ in $(Y, \eta)$ to the geometrical content of the same formula in $(X, \tau)$. If we apply this idea just to a single logical constant, it implies that the inverse image function should preserve all the logical constants in the language. And using the definition of the interpretation, it is pretty clear that for all the logical constants, except for the implication, this is the case. However, the case of the implication is totally different and it is in fact the source of the problems here. Let us explain this unexpected behavior by an example: let $f: \mathbb{R} \to \mathbb{R}$ as the following continuous function:
\[
f(x)=\begin{cases}
-x-1 & x \in (-\infty, -1)\\
0 & x \in [-1, -1]\\
x+1 & x \in (1, +\infty)\\
\end{cases} 
\]
and consider $U=(0, + \infty)$. Then, it is easy to compute that $f^{-1}(U \rightarrow \bot)=\emptyset$ and $f^{-1}(U) \rightarrow f^{-1}(\bot)=(-1, 1)$ which means 
\[
f^{-1}(U \rightarrow \bot) \neq f^{-1}(U) \rightarrow f^{-1}(\bot).
\]
Therefore, we can conclude that the geometric morphisms do not preserve the intuitionistic implication which implies that this constant is not genuinely geometric and we have to refine its definition if we want to have any meaningful geometric logic. There are different approaches to deal with this problem. The first one is the approach of geometric logic and it is based on the total elimination of the implication constant. From the logical point of view this approach seems extremely limiting and it leaves the logic without its usual power to formalize mathematical theories. However, in the real situations arising in topos theory it actually leads to deep understanding of the geometrical content of mathematical theories. The other approach is to keep the implication but weaken its power in a way that the geometric maps or at least some natural classes of them preserve this new implication. To implement this idea, we will extend the coherent logic with an implication which is weaker than the usual intuitionistic implication and we will show that it gets preserved under a natural class of continuous functions.\\

So far, we have investigated the geometrical case of implication. Let us now focus on the the second aspect of intuitionistic logic as a meta-theory for the notion of computability or more generally as a calculus of constructions. To do so, consider the BHK interpretation as the informal canonical interpretation of the intuitionistic constants and let us explain how it works. For simplicity, interpret any formula $A$ as the set of all of its proofs and write $a \in A$ as the proposition ``$a$ is a proof for $A$". Then:\\

For conjunction we have:
\[
A \wedge B=\{(b, c)| b \in A \; \text{and} \; c \in B\}
\]
which means that a construction $a$ is a proof for $A \wedge B$ iff $a$ equals to a pair $(b, c)$ where $b$ is a proof for $A$ and $c$ is a proof for $B$.\\
Similarly, for the disjunction we have
\[
A \vee B=\{(i, c)| \text{if} \; i=0 \; \text{then} \; b \in A \; \text{and} \; \text{if} \; i=1 \; \text{then} \; c \in B\}
\]
which means that $a$ is a proof for $A \vee B$ if $a$ is a pair that the first component says which of $A$ or $B$ is proved and the second component provides the corresponding proof.\\
And for the implication we have
\[
A \rightarrow B=\{f| \forall x \in A \rightarrow f(x) \in B \}
\]
which means that $f$ is a proof for $A \rightarrow B$ if $f$ provides a method to transform \textit{any} proof of $A$ to a proof of $B$.\\

The first and the second parts are simple and acceptable. The definition is recursive and to define the notion of proof, it is enough to just refer locally to the structure of the possible proof and the recursive notion of ``being a proof" that we are defining. But in the case of implication, the definition uses a universal quantifier on all possible constructions, including the proofs of $A \rightarrow B$ themselves, that we are defining at the moment. This definition in this present form is clearly impredicative. It is possible to see this impredicativity in a more technical way and in the concrete case of computability theory. To do so, consider the sister notion of realizabality and the case of realization of the elimination rule of the implication. To realize this rule, we clearly need the universal Turing machine to apply any arbitrary Turing machine, including itself, to all the possible inputs. There is no need to explain how and why this should be considered self-referential and hence impredicative. (For a historical and philosophical discussion on impredicativity of intuitionistic implication, See \cite{Van}.)\\

To deal with this impredicative definition, one reasonable approach is using the notion of time to control the order of constructing proofs to keep everything predicative. Roughly speaking, after adding time to the game we can revise the BHK interpretation in the following way: $f$ is a proof of $A \rightarrow B$ at the time instance $n$, if for any proof $x$ of $A$ appearing in time $m$ \textit{after} $n$, $f(x)$ is a proof for $B$ at the same time $m$. As it may be clear even from this rough explanation, formalizing this concept needs considering presheaves over some poset to capture the growing constructions over the structure of time. Moreover, and more importantly, we also need a suitable functor on our constructions to take care of the modality of time or more precisely the preposition ``after" in the definition. We will call these topoi, modal topoi and we will explain them in the last section of this paper.\\

And finally and as the third example, consider the Cartesian closed categories and their relationship with simply typed lambda calculus which can be considered as the higher and more structured propositional intuitionistic logic. It seems that these categories formalize the intuitive idea of being closed in category theory in the sense that they have an internal hom functor to simulate the behavior of their external Hom functor. The problem here is also the elimination rule which is just the evaluation map in this discourse. Consider the definition of the exponential objects. Intuitively, $B^A$ plays the role of  $\Hom(A, B)$ internally. Then the question is that what does it mean to have an arrow $ev: A \times B^A \to B$. Intuitively, it means that it applies all morphisms in $B^A$ to the elements of their source $A$. But in the categorical discipline we can not and should not refer to the elements of an object, first because it contradicts with the conceptual foundation of category theory and second, because it simply is meaningless in the case that the category is not a category of structured sets. In the best case, we can at most use generalized elements for this morphism and talk about something like $A^1 \times B^A \to B^1$ or more generally $A^C \times B^A \to B^C$. One can try to justify the evaluation map using these generalized elements and another assumption that says $A \simeq A^1$. What the assumption says is that the object $A$ has the same role in the category as $A^1$. It can happen in some cases that the category satisfies this condition as an extra feature but it is definitely not related to the notion of closedness. To have another intuition, think about it in the following way: $B^A$ internalizes the Hom structure of the category but this internalization is not necessarily related to the lower level of the real objects. Therefore, we can claim that the way that Cartesian closed categories capture the notion of closedness is not faithful and it has more hidden assumptions than what it reveals at the first glance. The solution for this problem is formalizing a new weaker type of exponential objects which mimic the behavior of the external Hom functor in a more faithful manner. This is what we call weak exponentials and we will define in the first following section.
\section{Weak Exponentials}
Let us first define the notion of a weak exponential. As we explained in the Introduction, it is supposed to formalize the concept of closedness without conceptual referring to any assumed internal structure of objects. The definition is a variant of Eilenberg-Kelly's definition of closed categories \cite{Eil}.
\begin{dfn}\label{t1-1}
Let $\mathcal{C}$ be a category with finite products and terminal object. A functor $[-, -] : \mathcal{C}^{op} \times \mathcal{C} \to \mathcal{C}$ together with the following data 
\begin{itemize}
\item[$(i)$]
a transformation $j_X : 1 \to [X, X]$, extranatural in $X$,
\item[$(ii)$]
a transformation $L_{YZ}^X : [X,Y] \times [Y,Z] \to [X,Z]$ natural in $X$ and $Z$ and extranatural in $Y$
\end{itemize}
is called weak exponential or weak internal hom if the following diagrams commute:

\begin{itemize}
\item[$(i)$]
For every $X$, and $Y$, $j_X$ plays the right unit role for the composition $L$:\\

\begin{center}
\begin{tikzpicture}
\node (C) {$1 \times [X, Y]$};
  \node (P) [below of=C] {$[X, X] \times [X, Y]$};
  \node (Ai) [right of=P, node distance=4cm] {$[X, Y]$};
  \draw[->] (C) to node [swap] {$\langle j_X, id_{[X, Y]} \rangle$} (P);
  \draw[->] (P) to node [swap] {$L^X_{XY}$} (Ai);
  \draw[->] (C) to node {$p_1 \circ id_{[X, Y]}$} (Ai);
\end{tikzpicture}
\end{center}

\item[$(ii)$]
For every $X$, and $Y$, $j_Y$ plays the left unit role for the composition $L$:\\

\begin{center}
\begin{tikzpicture}
\node (C) {$[X, Y] \times 1$};
  \node (P) [below of=C] {$[X, Y] \times [Y, Y]$};
  \node (Ai) [right of=P, node distance=4cm] {$[X, Y]$};
  \draw[->] (C) to node [swap] {$\langle id_{[X, Y]}, j_Y \rangle$} (P);
  \draw[->] (P) to node [swap] {$L^Y_{XY}$} (Ai);
  \draw[->] (C) to node {$p_0 \circ id_{[X, Y]}$} (Ai);
\end{tikzpicture}

\end{center}

\item[$(iii)$]
$L$ is associative, i.e. for every $X$, $Y$, $Z$ and $W$:\\

\begin{center}
\begin{tikzpicture}
  \node (P) {$[X, Y] \times [Y, Z] \times [Z, W]$};
  \node (B) [right of=P, node distance=6cm] {$[X,Z] \times [Z, W]$};
  \node (A) [below of=P] {$[X, Y] \times [Y, W]$};
  \node (C) [below of=B] {$[X, W]$};
  \draw[->] (P) to node {\tiny{$\langle  L^Y_{XZ} \langle p_0, p_1 \rangle, p_2 \rangle$}} (B);
  \draw[->] (P) to node [swap] {\tiny{$\langle p_0, L^Z_{YW} \langle p_1, p_2 \rangle \rangle$}} (A);
  \draw[->] (A) to node [swap] {\tiny{$L^Y_{XW}$}} (C);
  \draw[->] (B) to node {\tiny{$L^Z_{XW}$}} (C);
\end{tikzpicture}
\end{center}

\end{itemize} 
The category $\mathcal{C}$ equipped with a weak exponential is called strong. Moreover, if the map $\gamma :\mathcal{C}(X,Y) \to \mathcal{C}(1,[X,Y])$ defined by $f \mapsto [1,f](j_X)$ is a bijection, the strong category $\mathcal{C}$ is called weakly-closed. If there exists also a natural transformation $i_X: X \to [1, X] $, the category is called a Curry category and if $i$ is a natural isomorphism, it is called a closed category.
\end{dfn}
It is clear that being strong in the sense of Definition \ref{t1-1} means that the category is strong enough to internalize its Hom structure while weakly closedness is just the condition that the category is actually internalizes exactly the Hom structure and it does not add anything else to it.\\

There are some natural concrete examples of weak exponentials.
\begin{exam}\label{t1-2}
Let $\mathcal{C}$ be a Cartesian closed category and $F : \mathcal{C} \to \mathcal{C}$ be an arbitrary functor. Define $[X, Y]=\exp(FX, FY)$ and $[f, g]=\exp(Ff, Fg)$ and use $j$ and $L$ as the original natural transformations available in $\mathcal{C}$. 
\end{exam}
\begin{exam}\label{t1-3}
Consider the category $\Top$ of topological spaces and continuous functions. We know that $\Top$ is not Cartesian closed. However, there exists a natural canonical weak exponential object inside the category which internalizes some part of the structure of the category itself. Suppose $\Top_{\CHaus}$ is the subcategory of all compact Hausdorff spaces, $\beta : \Top \to \Top_{\CHaus}$ is a Stone-Cech compactification functor and $U:   \Top_{\CHaus} \to \Top$ is the forgetful functor. We know that $\beta \dashv U$. Since all compact Hausdorff spaces are exponentiable we can define $[X, Y]=U[(\beta Y)^{\beta X}]$ and consider $i_X: X \to U(\beta X)^1=[1, X]$ as the composition of the unit of the compactification adjunction and the isomorphism between $\beta X$ and $(\beta X)^1$. Hence $\Top$ with this structure is a Curry category.
\end{exam}
\begin{exam}\label{t1-4}
Let $\mathbb{C}$ be a small category and $I: \mathbb{C} \to \mathbb{C}$ be an arbitrary functor. Then consider the category of presheaves $\mathcal{E}=\Set^{\mathbb{C}^{op}}$ and $[-, -]: \mathcal{E}^{op} \times \mathcal{E} \to \mathcal{E}$ as 
\[
[E, F](c)=\Hom(\mathbf{y}(I(c)) \times E, F)
\]
and $[\alpha, \beta](c)(f)=\beta f (id \times \alpha)$. It is not hard to observe that this functor is a weak exponential using the most natural $j$ and $L$. For some concrete useful instance of this presheaf construction, put $\mathbb{C}=(\mathbb{N}, \leq)$ as the usual poset category of natural numbers and $I: \mathbb{N} \to \mathbb{N}$ as the functor $i(n)=n+1$. 
\end{exam}
\section{Modal Spaces}
In this section we will define the notion of modal space as an abstract setting which provides a natural candidate for predicative implication. For this purpose, let us begin with a concrete model to explain how the notion of time can be effective in defining predicative implications and then we will generalize the setting to the more abstract situations.\\

Assume that $(\mathbb{N}, \leq)$ is the usual set of natural numbers with its usual order as the formalization of the notion of time and also assume that we are interested in the class of all increasing functions $f: \mathbb{N} \to \{0, 1\}$ as the class of variable statements which their truth values vary in time. First of all, note that this situation is a loose version of what we explained in the Introduction in which we replaced constructions by truth values. The reason is that this truth value type of investigation is more appropriate in the case of propositional statements that just have truth value. To speak about constructions we have to go one level up to work with simply typed lambda calculus which is actually what we intend to do in the last section of this paper. Secondly, it is pretty clear that there exist natural interpretations for conjunction and disjunction in this discourse. It is enough to interpret conjunction as the the pointwise minimum of functions and disjunction as the pointwise maximum. For implication, if we follow the intuition that we explained before, we will have the following definition: $[f \rightarrow g](n)=1$ iff for all $x>n$, $f(x) \leq g(x)$ which means that the statement $f \rightarrow g$ is true at $n$ if for every instance $x$ \textit{after} $n$, $f(x)$ implies $g(x)$. It seems that we have had a satisfactory formalization of all propositional constants and specifically the predicative implication. But there is a problem with this formalization. If we investigate our construction carefully, we will observe that what we used was just the preposition \textit{after} and not the whole power of the concept of time which more or less determines every details of the whole setting and then leaves no room for generalization. Therefore, the natural question now is that is there any generalized setting in which we can implement the construction of implication? In other words, is it possible to find an abstract version of the preposition \textit{after} without referring to the notion of time? The answer to this question is positive and we will explain how.\\

First let us import some topological language to simplify the situation. Put the order topology on the ordered structure $(\mathbb{N}, \leq)$ in which opens are just the upward-closed subsets. Then do the same thing for the set of truth values $\{0, 1\}$ to have the Sierpinski space. Then an increasing function $f: \mathbb{N} \to \{0, 1\}$ is nothing but a continuous function and since in this situation a continuous function is uniquely determined with its support $f^{-1}(1)$, we can change our perspective to consider open subsets instead of increasing functions and interpret any open $U$ as the points in time that a sentence is true. Now the conjunction and the disjunction become intersection and union which are obviously simpler than their original form. For implication though, we will have the following: $n \in U \rightarrow V$ iff for all $x> n$, if $x \in U$ then $x \in V$. If we represent any element $m$ by its canonical open $U_m=\{y | y \geq m \}$ then we have
$U_n \subseteq U \rightarrow V$ iff for all $x>n$, $U_x \cap U \subseteq V$. As we guessed before, the only thing to do now is just finding a formalization for $x>n$. For this purpose, define the operator $J$ on open subsets in the following way: 
\[
JW=\{z| \exists y \leq z \wedge y \in W \}
\]
if we come back to the original definition with functions it is equivalent to define $Jf$ as $Jf(n)=max\{f(m)| m \leq n\dot{-}1\}=f(n\dot{-}1)$. Intuitively, $Jf$ is just the pulling back version of $f$ over the line of time. Now with this operator at hand, we know that $x>n$ iff $U_x \subseteq JU_n$. Therefore, it is easy to see that we can define $U \rightarrow V$ as 
\[
\bigcup\{W | JW \cap U \subseteq V\}.
\]
This new representation of the predicative implication can be easily lifted to any topological space and since we did not mention the internal structure of opens it is possible to generalize it even to locales. But what about the operator $J$ which is responsible for the effect of time? It is clear that we have to put some conditions on this operator to have a meaningful interpretation. First note that $J$ is increasing. Second and more importantly, $J$ preserves unions. The reason is that computing the pointwise maximum of functions and also computing the pulling back operator $J$, which is based on computing maximums, commute. It is also possible to prove that $J$ preserves intersections in our setting with linearly ordered time or even in the case that we use trees instead of lines. However, in the more general case this is not necessarily true and as the following investigations show it is not even needed. Finally, our $J$ has the property $JU \subseteq U$ which means that the notion of truth is cumulative and preserved by time. However, we also relax this condition in our general definition to keep the commitments of our topological semantics as minimal as possible. But we will investigate spaces with this condition which we will call temporal spaces later. Now we have:
\begin{dfn}\label{t2-1}
A pair $(X, J)$ is called a modal space if $X$ is a locale and $J: X \to X$ is a monotone function which preserves all joins.
\end{dfn}
\begin{rem}
Note that if we interpret the notion of topology as a covering property, as it is usual in formal topology, then the condition of being join preserving means that $J$ respects the topology because if $\bigcup_i U_i=U$ then $\bigcup_i J(U_i)=J(U)$.
\end{rem}
Let us illuminate the above definition by some examples:
\begin{exam}\label{t2-2}
Assume that $(X, \tau)$ is a topological space and $f: X \to X$ is a continuous function, then $(X, \tau, f^{-1})$ is a modal space. The reason simply is that $f^{-1}$ is monotone and union preserving.
\end{exam}
The following two are the generalizations of the setting that we began with. They will play an important role in the rest of the paper.
\begin{exam}\label{t2-3}
Assume that $(W, R)$ is a relational frame such that $R \subseteq W \times W$ and $J: P(W) \to P(W)$ as $J(U)=\{x | \exists y\; R(y, x) \wedge y \in U \}$. $J$ is trivially monotone and join preserving. Therefore, $(P(W), J)$ is a modal space. 
\end{exam}

\begin{exam}\label{t2-4}
Assume that $(W, R)$ is a relational frame such that $R \subseteq W \times W$ is a transitive relation. Then define $\tau$ as the set of all upward closed subsets of $W$. $\tau$ is trivially a topology. Define $J: \tau \to \tau$ as $J(U)=\{x | \exists y\; R(y, x) \wedge y \in U \}$. First of all, $J$ sends any upward-closed $U$ to an upward closed set since $R$ is transitive. Secondly, $J$ is trivially monotone and join preserving. Therefore, $((W, \tau), J)$ is a modal space. 
\end{exam}
The following shows that any modal space is equipped with its natural weak implication:
\begin{thm}\label{t2-5}
Let $(X, J)$ be a modal space. It is possible to define $\rightarrow : X \times X \to X$ such that for any $a, b, c \in X$, $Jc \wedge a \leq b$ iff $c \leq a \rightarrow b$. In categorical terms, $J(-) \times a$ is a left adjoint to $a \rightarrow (-)$ for any $a \in X$.
\end{thm}
\begin{proof}
Define $a \rightarrow b=\bigvee \{d | Jd \wedge a \leq b \}$. It is clear that if $Jc \wedge a \leq b$ then by definition $c \leq a \rightarrow b$. For the converse, if $c \leq a \rightarrow b$, then $c \leq \bigvee \{d | Jd \wedge a \leq b \}$ hence $Jc \leq J\bigvee \{d | Jd \wedge a \leq b \}$ since $J$ preserve all joins, $Jc \leq \bigvee \{Jd | Jd \wedge a \leq b \}$. Therefore, $Jc \wedge a \leq \bigvee \{Jd | Jd \wedge a \leq b \} \wedge a$ by distributivity, $Jc \wedge a \leq \bigvee \{Jd \wedge a | Jd \wedge a \leq b \} \leq b$. Hence $Jc \wedge a \leq b$. 
\end{proof}
This weak exponential is well-behaved under the application of continuous functions, i.e.,
\begin{thm}\label{t2-6}
Assume $X$ is a locale, $(Y, J)$ is a modal space and $f:X \to Y$ is a continuous surjection such that $f^{-1}$ has a left adjoint. Then there exists $I: X \to X$ such that $(X, I)$ is a modal space and $f^{-1}(a \rightarrow_J b)=f^{-1}(a) \rightarrow_I f^{-1}(b)$.
\end{thm}
\begin{proof}
Define $I(x)=f^{-1}Jf_!(x)$ where $f_!$ is the left adjoint for $f^{-1}$. Since $f_!$ is a left adjoint, it is monotone and preserves all joins hence $I$ is also monotone and join preserving. Moreover, we know that $f^{-1}(a \rightarrow_J b)$ is equal to $f^{-1}(\bigvee \{c | Jc \wedge a \leq b \})$. Since $f$ is continuous, $f^{-1}$ preserves all joins and hence $f^{-1}(a \rightarrow_J b)$ is equal to $\bigvee \{f^{-1}(c) | Jc \wedge a \leq b \}$. Since $f$ is surjective, $f^{-1}$ is one to one, hence $Jc \wedge a \leq b$ iff $f^{-1}(Jc \wedge a) \leq f^{-1}(b)$. Since $f$ is continuous, this is equivalent to $f^{-1}(Jc) \wedge f^{-1}(a) \leq f^{-1}(b)$. On the other hand $f^{-1}(a) \rightarrow_I f^{-1}(b)$ is equivalent to $\bigvee \{d | Id \wedge f^{-1}(a) \leq f^{-1}(b) \}$. Hence we have to show that 
\[
\bigvee \{f^{-1}(c) | f^{-1}(Jc) \wedge f^{-1}(a) \leq f^{-1}(b) \}=\bigvee \{d | Id \wedge f^{-1}(a) \leq f^{-1}(b) \}.
\]
For any $c$, pick $d=f^{-1}(c)$. Since $d \leq f^{-1}(c)$, $f_!(d) \leq c$, hence $Jf_!(d) \leq Jc$ and then $f^{-1}Jf_!(d) \leq f^{-1}(Jc)$ therefore $Id \leq f^{-1}(Jc)$ which means that $f^{-1}(Jc) \wedge f^{-1}(a) \leq f^{-1}(b)$ implies $Id \wedge f^{-1}(a) \leq f^{-1}(b)$, therefore 
\[
\bigvee \{f^{-1}(c) | f^{-1}(Jc) \wedge f^{-1}(a) \leq f^{-1}(b) \} \leq \bigvee \{d | Id \wedge f^{-1}(a) \leq f^{-1}(b) \}.
\]
For the converse, if we have $d$ such that $Id \wedge f^{-1}(a) \leq f^{-1}(b)$, define $c=f_!(d)$. First of all $f^{-1}(Jc)=f^{-1}Jf_!(d)$. Secondly, since $f_!d \leq c$ we have $d \leq f^{-1}(c)$ which implies
\[
\bigvee \{f^{-1}(c) | f^{-1}(Jc) \wedge f^{-1}(a) \leq f^{-1}(b) \} \geq \bigvee \{d | Id \wedge f^{-1}(a) \leq f^{-1}(b) \}
\]
which completes the proof.
\end{proof}
The useful corollary of the Theorem \ref{t2-6} is:
\begin{cor}\label{t2-7}
Assume $(Y, \sigma)$ is an Alexandrov space, $J: \sigma \to \sigma$ is a monotone function and $f:(X, \tau) \to (Y, \sigma)$ is a continuous surjection. Then there exists $I: \tau \to \tau$ such that $(X, \tau, I)$ is a modal space and $f^{-1}(a \rightarrow_J b)=f^{-1}(a) \rightarrow_I f^{-1}(b)$.
\end{cor}
\begin{proof}
Define $f_!(U)=\bigcap_{f(U) \subseteq W} W$ where $W$ is an open subset of $Y$. $f_!$ is well-defined because the intersection of open subsets is also open. Moreover we have $f_! \dashv f^{-1}$ because if $U \subseteq f^{-1}(V)$ then $f(U) \subseteq V$ and hence $V \in \bigcap_{f(U) \subseteq W} W$ and therefore $f_!(U) \subseteq V$. And conversely, if $f_!(U) \subseteq V$ then $\bigcap_{f(U) \subseteq W} W \subseteq V$ hence $\bigcap_{f(U) \subseteq W} f^{-1}(W) \subseteq f^{-1}(V)$. Since for any $f(U) \subseteq W$, $U \subseteq f^{-1}(W)$, hence $U \subseteq \bigcap_{f(U) \subseteq W} f^{-1}(W)$ which implies $U \subseteq f^{-1}(V)$. Hence $f_! \dashv f^{-1}$ and finally by the Theorem \ref{t2-6} we can prove the claim.
\end{proof}
As a conclusion for this section, let us specify some important classes of modal spaces.
\begin{dfn}\label{t2-8}
\begin{itemize}
\item[$(i)$]
The class $\mathbf{MS}$ consists of all modal spaces.
\item[$(ii)$]
A modal space is called semi-cotemporal if $Ja=0$ implies $a=0$. Denote the set of these spaces by $\mathbf{sCoTS}$.
\item[$(iii)$]
A modal space is called semi-temporal if $JJ(a) \leq J(a)$. Denote the set of these spaces by $\mathbf{sTS}$.
\item[$(iv)$]
A modal space is called temporal if $J(a) \leq a$. Denote the set of these spaces by $\mathbf{TS}$.
\item[$(v)$]
A modal space is called cotemporal if $ a \leq J(a)$. Denote the set of these spaces by $\mathbf{CoTS}$.
\end{itemize}

Moreover, by $\mathbf{sS}$ we mean $\mathbf{sCoTS} \cap \mathbf{TS}$ and by $\mathbf{S}$ we mean $\mathbf{TS} \cap \mathbf{T}$ and when we put $b$ in the left hand-side of the name of the class, we mean the subclass of all boolean modal spaces.
\end{dfn}
The following theorem shows how modal spaces are related to the categories with weak exponentials:
\begin{thm}\label{t2-9}
Let $(X, J)$ be a modal space, then
\begin{itemize}
\item[$(i)$]
$X$ equipped with $\rightarrow_J$ is a strong category.
\item[$(ii)$]
$X$ equipped with $\rightarrow_J$ is weakly closed iff $J(1)=1$.
\item[$(iii)$]
$X$ equipped with $\rightarrow_J$ is Curry category iff $(X, J)$ is temporal.
\end{itemize}
\end{thm}
\begin{proof}
For $(i)$, we have to show firstly that $a \to_J b $ is decreasing in $a$ and increasing in $b$, secondly $1 \leq a \to_J a$ and thirdly the following:
\[
(a \to_J b) \wedge (b \to_J c) \leq (a \to_J c) .
\]
The first condition is clear by the definition of $a \rightarrow_J b$ and the fact that $J$ is monotone. The second is clear by the Theorem \ref{t2-5}, because $1 \leq a \rightarrow_J a$ iff $J1 \wedge a \leq a$ which is the case. For the third one, again by using the Theorem \ref{t2-5}, it is enough to check $J[(a \to_J b) \wedge (b \to_J c)] \wedge a \leq c $. But since $J$ is increasing, 
\[
J[(a \to_J b) \wedge (b \to_J c)] \leq J(a \to_J b) \wedge J(b \to_J c).
\]
Now taking meet with $a$ and using the facts that $a \wedge J(a \to_J b) \leq b$ and $b \wedge J(b \rightarrow_J c) \leq c$, we will have
\[
J[(a \to_J b) \wedge (b \to_J c)] \wedge a \leq c.
\]

For $(ii)$, notice that $X$ is weakly closed if $a \to_J b=1$ implies $a \leq b$. Hence, if $J(1)=1$, then by Theorem \ref{t2-5}, $1 \leq a \to_J b$ implies $J(1) \wedge a \leq b$ which implies $a \leq b$. Conversely, assume that the category is weakly closed. We have $1 \to_J J(1)=1$ because 
\[
1 \to_J J(1)=\bigvee_{Jc \leq J(1)} c=1.
\]
Hence by weakly closedness $1 \leq J(1)$ which means that $J(1)=1$.\\

For $(iii)$, note that the category is Curry if for any $a$, $a \leq 1 \to_J a$ which is equivalent to $Ja \leq a$, by Theorem \ref{t2-5}.
\end{proof}

\section{Modal Logics and Topological Semantics}
In this section we will show how modal spaces make a sound and complete topological interpretation for some basic modal logics. But first we need to define the boolean modal spaces to interpret the classical logic in the base of the classical modal systems.
\begin{dfn}\label{t3-1}
A modal space $(X, J)$ is called boolean if the locale $X$ is boolean, i.e. for any $a \in X$ there exists an element $b \in X$ (a complement for $a$) such that $a \wedge b=0$ and $a \vee b=1$.
\end{dfn}
\begin{rem}\label{t3-2}
Note that the complement is trivially unique because if $b$ and $c$ are both complements of $a$ then by definition $(a \vee b) \wedge c=1 \wedge c=c$ but by distributivity $(a \vee b) \wedge c=(a \wedge c) \vee (b \wedge c)=b \wedge c$. Hence $b \wedge c=c$. By symmetry we have $b \wedge c=b$ and hence $b=c$. Therefore the complement is unique and we can denote it by $\neg a$.
\end{rem}
Now we are ready to expand the well-known relationship between topological spaces and the extensions of the modal logic $\mathbf{S4}$ to the extensions of $\mathbf{K}$ and the boolean modal spaces. This shows how the language of modal logic can serve as the language for boolean modal spaces. Later, we will show the same kind of relationship between the sub-intuitionistic logics and the modal spaces without the boolean condition.
\begin{dfn}\label{t3-3}
A topological model for modal logics is a tuple $(X, J, V)$ such that $X$ is a boolean locale, $(X, J)$ is a modal space and $V:\mathcal{L}_{\Box} \to X$ is a valuation function such that: 
\begin{itemize}
\item[$(i)$]
$V(\top)=1$ and $V(\bot)=0$.
\item[$(ii)$]
$V(A \wedge B)=V(A) \wedge V(B)$.
\item[$(iii)$]
$V(A \vee B)=V(A) \vee V(B)$.
\item[$(iv)$]
$V(A \rightarrow B)=\neg V(A) \vee V(B)$.
\item[$(v)$]
$V(\neg A)=\neg V(A)$.
\item[$(vi)$]
$V(\Box A)=\bigvee \{c | Jc \leq V(A) \}=1 \rightarrow_J V(A)$.
\end{itemize}
We say $(X, J, V) \vDash \Gamma \Rightarrow A$ when $\bigwedge_{\gamma \in \Gamma} V(\gamma) \leq V(A)$ and $(X, J) \vDash \Gamma \Rightarrow A$ when for all $V$, $(X, J, V) \vDash \Gamma \Rightarrow A$. Moreover, if for some fix $X$ and for all $(X, J)$ in some class $\mathcal{C}$ we have $(X, J) \vDash \Gamma \Rightarrow A$, then we write $X \vDash_{\mathcal{C}} \Gamma \Rightarrow A$. Also when $\Gamma$ is empty, we omit the symbol $\Rightarrow$ in all of the situations.
\end{dfn}
The following examples show how the modal spaces can be used to refute some modal statements.
\begin{exam}\label{t3-4}
Let $X$ be a set and $f$ a function on $X$. Define $X_{f}=(P(X), J_f)$ where $J_f : P(X) \to P(X)$ is the inverse image function, i.e. $J_f(U)=f^{-1}(U)$. It is clear that $P(X)$ is boolean and $J$ is monotone and join preserving. It is also easy to see that for any modal formula $A$, we have $V(\Box A)= \bigcup\{U | f^{-1}(U) \subseteq V(A)\}=f(V(A))$. Now assume that there exists $U \subseteq X$ such that $f(U) \nsubseteq f(f((U))$ then if we define $V(p)=U$ we have $V(\Box p) \nsubseteq V(\Box \Box p)$ which implies $X_f \nvDash \Box p \rightarrow \Box \Box p$.
\end{exam}

\begin{exam}\label{t3-5}
Let $X$ be a set and consider the model $X_{A}=(P(X), J_A)$ where $J_A : P(X) \to P(X)$ is the following function: $J_A(U)=U \cap A$. It is clear that $P(X)$ is boolean and $J$ is monotone and join preserving. Now, we want to show that $X_A \nvDash \Box p \rightarrow p$. Pick $V(p)=B$ such that $A^c \nsubseteq B$. Then we have $V(\Box p)= \bigcup\{U | U \cap A \subseteq B\}=A^c \cup B$. Hence, $V(\Box p)=A^c \cup B \nsubseteq B=V(p)$.
\end{exam}

In the following we will show the soundness of the topological semantics:
\begin{thm}\label{t3-6}(Soundness Theorem)
\begin{itemize}
\item[$(i)$]
If $\Gamma \vdash_{\mathbf{K}} A$, then $ \Gamma \vDash_{\mathbf{bMS}} A$.
\item[$(ii)$]
If $\Gamma \vdash_{\mathbf{D}} A$, then $ \Gamma \vDash_{\mathbf{bsCoTS}} A$.
\item[$(iii)$]
If $\Gamma \vdash_{\mathbf{T}} A$, then $ \Gamma \vDash_{\mathbf{bCoT}} A$.
\item[$(iv)$]
If $\Gamma \vdash_{\mathbf{K4}} A$, then $ \Gamma \vDash_{\mathbf{bsTS}} A$.
\item[$(v)$]
If $\Gamma \vdash_{\mathbf{KD4}} A$, then $ \Gamma \vDash_{\mathbf{bsS}} A$.
\item[$(vi)$]
If $\Gamma \vdash_{\mathbf{S4}} A$, then $ \Gamma \vDash_{\mathbf{bS}} A$.
\end{itemize}
\end{thm}
\begin{proof}
First of all it is clear that a proof for the case $\Gamma=\emptyset$ suffices. Now, by induction on the length of the proof of $A$ in the corresponding system, we will prove that for any corresponding modal space $(X, J)$ and any $V$, $(X, J, V) \vDash A$. If $A$ is a classical tautology then the claim is clear since $X$ is a boolean locale.\\
For $(i)$, if $A$ is an instance of the axiom $\mathbf{K}$ then we have to show 
\[
V(\Box (B \rightarrow C) \rightarrow V(\Box B \rightarrow \Box C) )=1.
\]
To this end, it is enough to show 
\[
V(\Box (B \rightarrow C)) \wedge V(\Box B) \leq V(\Box C).
\]
Assume $J(a) \leq \neg V(B) \vee V(C)$ and $J(b) \leq V(B)$, then $J(a) \wedge J(b) \leq V(C)$. Since $J$ is monotone, $J(a \wedge b) \leq J(a) \wedge J(b)$, hence, $J(a \wedge b) \leq C$, therefore for any generating $a$ for $V(\Box (B \rightarrow C))$ and any generating $b$ for $V(B)$, $a \wedge b$ is a generating element for $C$, hence $V(\Box (B \rightarrow C)) \wedge V(\Box B) \leq V(\Box C)$.\\
For $(ii)$, if $A$ is the axiom $\mathbf{D}=\neg \Box \bot$, then $V(\Box \bot)=\bigvee \{a | Ja \leq 0 \}=0$ since the modal space is semi-cotemporal.\\
For $(iii)$, if $A$ is an instance of the axiom $\mathbf{T}$, then we have to show $V(\Box B) \leq V(B)$. Assume that $a$ is a generating element for $V(\Box B)$, then $J(a) \leq V(B)$. Since the modal space is cotemporal, we have $a \leq J(a)$. Hence we conclude $a \leq V(B)$ which completes the proof.\\
For $(iv)$, if $A$ is an instance of the axiom $\mathbf{4}$, then we have to show $V(\Box B) \leq V(\Box \Box B)$. Assume that $a$ is a generating element for $V(\Box B)$, then $J(a) \leq V(B)$. First we show $J(a) \leq V(\Box B)$. Since the modal space is semi-temporal, we have $JJ(a) \leq J(a)$ and thus $JJ(a) \leq V(B)$. Hence, $JJ(a)$ is a generating element for $V(\Box B)$, therefore $J(a) \leq V(\Box B)$. Now, we know that $a$ is a generating element for $V(\Box \Box B)$, hence $V(\Box B) \leq V(\Box \Box B)$.\\ 

For the rules, if $A$ is proved by the modus ponens rule then the claim is easy to prove. And if $A$ is proved by the necessitation rule, we have $V(B)=1$ then $V(\Box B)=\bigcup \{a | J(a) \leq V(B)=1 \}=1$ which is what we wanted.
\end{proof}
Moreover, we can prove the completeness of these topological models:
\begin{thm}\label{t3-7}(Completeness Theorem)
\begin{itemize}
\item[$(i)$]
If $\Gamma \vDash_{\mathbf{bMS}} A$, then $ \Gamma \vdash_{\mathbf{K}} A$.
\item[$(ii)$]
If $\Gamma \vDash_{\mathbf{bsCoTS}} A$, then $ \Gamma \vdash_{\mathbf{D}} A$.
\item[$(iii)$]
If $\Gamma \vDash_{\mathbf{bCoT}} A$, then $ \Gamma \vdash_{\mathbf{T}} A$.
\item[$(iv)$]
If $\Gamma \vDash_{\mathbf{bsTS}} A$, then $ \Gamma \vdash_{\mathbf{K4}} A$.
\item[$(v)$]
If $\Gamma \vDash_{\mathbf{bsS}} A$, then $ \Gamma \vdash_{\mathbf{KD4}} A$.
\item[$(vi)$]
If $\Gamma \vDash_{\mathbf{bS}} A$, then $ \Gamma \vdash_{\mathbf{S4}} A$.
\end{itemize}
\end{thm}
\begin{proof}
To prove the completeness we will use the Kripke completeness for the modal systems. It is enough to consider the relational frame $(W, R, V)$ as a modal space $(P(W), \tau, V')$ as in the Example \ref{t2-3} where $V'=V$. By induction on the complexity of the formula $A$, we want to prove $w \in V'(A)$ iff $w \vDash A$. The atomic case is trivial by definition. The cases for conjucntion, disjunction and implication are also trivial. For the modal case, assume $w \in V'(\Box A)=\bigcup_{JU \subseteq V'(A)} U$. Therefore, there exists $U$ such that $w \in U$ and $JU \subseteq V'(A)$. Hence if $(w, u) \in R$, then $u \in JU$ and thus $u \in V'(A)$. By IH, $u \vDash A$. Hence $w \vDash \Box A$.
Conversely, if $w \vDash \Box A$ then for all $u$ such that $(w, u) \in R$ we have $u \vDash A$ which by using IH implies that $u \in V'(A)$. Therefore if we define $U=\{w\}$ then by the definition of $J$ we have $JU=\{u| (w, u) \in R\} \subseteq V'(A)$. Hence $w \in V'(\Box A)$.\\

The only thing that remains to prove is the compatibility conditions:\\

If $(W, R)$ is serial, then if $J(U)=\emptyset$ then $U=\emptyset$. Because if $x \in U$, then since $R$ is serial, there exists $y$ such that $(x, y) \in R$. Hence by definition $y \in J(U)=\emptyset$ which is a contradiction.\\
If $(W, R)$ is transitive, then $JJ(U) \subseteq J(U)$ because if $x \in JJ(U)$, then there exists $y \in J(U)$ such that $(y, x) \in R$. By the same line of reasoning, there exists $z \in U$ such that $(z, y) \in R$. Since $R$ is transitive, $(z, x) \in R$, therefore $x \in J(U)$.\\ 
If $(W, R)$ is reflexive, then $U \subseteq J(U)$ because if $x \in U$, then $(x, x) \in R$ hence $x \in J(U)$. 
\end{proof}
In the rest of this section we will try to generalize the completeness theorem to a more powerful version. In that version we fix the topological space and we will show how a modal logic can be considered as the modal logic of just one space.
\begin{thm}\label{t3-8}
Assume $X$ and $Y$ are boolean spaces and $f: X \to Y$ is a continuous surjection. Then for any modal formula $A$, if $X \vDash_{\mathcal{C}} A$ then $Y \vDash_{\mathcal{C}} A$ for any class $\mathcal{C}$ as in the Definition \ref{t2-8}.
\end{thm}
\begin{proof}
Let $V$ be a valuation for $(Y, I)$. Define $V'(p)=f^{-1}(V(p))$. Since any boolean locale is Alexandrov, by Theorem \ref{t2-7} we can prove that there exist $J$ such that $(X, J)$ is a modal space and $f^{-1}(a \rightarrow_I b)=f^{-1}(a) \rightarrow_J f^{-1}(b)$. It is easy to prove by induction on the complexity of $A$ that $V'(A)=f^{-1}(V(A))$.\\
For the atomic $A=p$ it is clear by the definition of $V'$. For $A=\bot, \top$, we have to show that $f^{-1}(0)=0$ and $f^{-1}(1)=1$. Both are correct since $f^{-1}$ is a right and also left adjoint and hence preserves limits and colimits.\\
The conjunction and disjunction parts are proved by the fact that $f^{-1}$ preserves meet and join because it is continuous. The proof of the modal part is implied by the fact that $f^{-1}(a \rightarrow_I b)=f^{-1}(a) \rightarrow_J f^{-1}(b)$ and $f^{-1}(1)=1$.\\

Therefore for any $A$ we have, $V'(A)=f^{-1}(V(A))$. Since $X \vDash A$, we have $(X, J, V') \vDash A$ hence $V'(A)=1$, therefore we have $f^{-1}(V(A))=1$. On the other hand we know that if $f^{-1}(c)=1$ then $c=1$ because if $f^{-1}(c)=1$ since $f^{-1}(1)=1$ we know $f^{-1}(c)=f^{-1}(1)$ but $f$ is a surjection which means that $f^{-1}$ is one to one, hence $c=1$. By this consideration and the fact that $f^{-1}(V(A))=1$ we can deduce $V(A)=1$ which completes the proof.\\
The last thing to prove is that if $I$ has one of the properties in the Definition \ref{t2-8} the corresponding $J$ has it as well. If $\mathcal{C}=\mathbf{sCoTS}$, then if $J(a)=0$ we have $f^{-1}If_{!}(a)=0$ hence $If_{!}(a)=0$ therefore $f_{!}(a)=0$. Since $f_{!}(a) \leq 0$, we have $a \leq f^{-1}(0)=0$ hence $a=0$.\\
If $\mathcal{C}=\mathbf{sTS}$, we want to show that $JJ(a) \leq J(a)$. Therefore we have to show $f^{-1}If_{!}f^{-1}If_{!}(a) \leq f^{-1}If_{!}(a)$. Since $f_{!}f^{-1}(c) \leq c$ for any $c$, $f_{!}f^{-1}If_{!}(a) \leq If_{!}(a)$ therefore $If_{!}f^{-1}If_{!}(a) \leq IIf_{!}(a)$. Since $IIf_{!}(a) \leq If_{!}(a)$, we have $If_{!}f^{-1}If_{!}(a) \leq If_{!}(a)$ which implies $f^{-1}If_{!}f^{-1}If_{!}(a) \leq f^{-1}If_{!}(a)$.\\
If $\mathcal{C}=\mathbf{CoTS}$, we want to show that $a \leq J(a)$. It is equivalent to $a \leq f^{-1}If_{!}(a)$ which is also equivalent to $f_{!}(a) \leq If_{!}(a)$ which what we have.
\end{proof}

\begin{cor}\label{t3-9}(Completeness Theorem, Strong version) For any topological space with infinitely many connected components we have:
\begin{itemize}
\item[$(i)$]
If $X \vDash_{\mathbf{bMS}} \Gamma \Rightarrow A$ then $\Gamma \vdash_{\mathbf{K}} A$.
\item[$(ii)$]
If $ X \vDash_{\mathbf{bsCoTS}} \Gamma \Rightarrow A$ then $\Gamma \vdash_{\mathbf{D}} A$.
\item[$(iii)$]
If $ X \vDash_{\mathbf{bCoTS}} \Gamma \Rightarrow A$ then $\Gamma \vdash_{\mathbf{T}} A$.
\item[$(iv)$]
If $ X \vDash_{\mathbf{bsTS}} \Gamma \Rightarrow A$ then $\Gamma \vdash_{\mathbf{K4}} A$.
\item[$(v)$]
If $ X \vDash_{\mathbf{bsS}} \Gamma \Rightarrow A$ then $\Gamma \vdash_{\mathbf{KD4}} A$.
\item[$(vi)$]
If $ X \vDash_{\mathbf{bS}} \Gamma \Rightarrow A$ then $\Gamma \vdash_{\mathbf{S4}} A$.
\end{itemize}
\end{cor}
\begin{proof}
The reason is that the modal spaces constructed by the finite relational frames are enough for completeness and for any of them like $(P(W), J_R)$, it is possible to find a surjective continuous function $f: X \to W$.
\end{proof}

\section{$J$-Logics and Sub-intuitionistic Logics}
So far, we have seen how modal logics can be used to describe the geometrical nature of modal spaces. However, modal systems do not provide a good syntactical reflection of the geometrical situation that we face in modal spaces. Therefore, it will be convenient to develop some propositional systems to have a more faithful language to formalize modal spaces. Moreover and after introducing these systems, we will show how some sub-intuitionistic logics can be embedded in these systems and how consequently the modal spaces can be considered as the natural topological interpretation of these weak logics.\\

Let $\mathcal{L}_J$ be the usual language of propositional logic with a unary modal operator $J$. To introduce some formal systems in this language, consider the following set of natural deduction rules:

\begin{flushleft}
 		\textbf{Structural Rules:}
\end{flushleft}
\begin{center}
  	\begin{tabular}{c c}
		\AxiomC{$ \Gamma \vdash A$}
		\LeftLabel{\tiny{$F$}}
		\UnaryInfC{$J\Gamma \vdash JA$}
		\DisplayProof
		 &\;
		\AxiomC{$ \Gamma_0 \vdash A$}
		\AxiomC{$ \Gamma_1, A \vdash B$}
		\LeftLabel{\tiny{$cut$}}
		\BinaryInfC{$\Gamma_0, \Gamma_1 \vdash B$}
		\DisplayProof
		\end{tabular}
\end{center}

\begin{flushleft}
  		\textbf{Propositional Rules:}
\end{flushleft}
\begin{center}
  	\begin{tabular}{c c}
  	
  	    \AxiomC{$ $}
		\LeftLabel{\tiny{$\top$}}
		\UnaryInfC{$\Gamma \vdash \top$}
		\DisplayProof
  	    &\;\;
		\AxiomC{$ \Gamma \vdash \bot$}
		\LeftLabel{\tiny{$\bot$}}
		\UnaryInfC{$\Gamma \vdash A$}
		\DisplayProof
  	    \\[3 ex]
  	
  		\AxiomC{$\Gamma_0, A  \vdash C $}
  		\AxiomC{$\Gamma_1, B  \vdash C$}
  		\LeftLabel{{\tiny $\vee E$}} 
  		\BinaryInfC{$ \Gamma_0, \Gamma_1, A \vee B \vdash C$}
  		\DisplayProof
	  		&
	   	\AxiomC{$ \Gamma \vdash A_i$}
   		\RightLabel{{\tiny $ (i=0, 1) $}}
   		\LeftLabel{{\tiny $\vee I$}} 
   		\UnaryInfC{$ \Gamma \vdash A_0 \lor A_1$}
   		\DisplayProof
	   		\\[3 ex]
   		\AxiomC{$ \Gamma \vdash A_0 \wedge A_1  $}
   		\RightLabel{{\tiny $ (i=0, 1) $}} 
   		\LeftLabel{{\tiny $\wedge E$}}  		
   		\UnaryInfC{$ \Gamma \vdash A_i $}
   		\DisplayProof
	   		&
   		\AxiomC{$\Gamma_0  \vdash A$}
   		\AxiomC{$\Gamma_1  \vdash B$}
   		\LeftLabel{{\tiny $\wedge I$}} 
   		\BinaryInfC{$ \Gamma_0, \Gamma_1 \vdash A \wedge B $}
   		\DisplayProof
   			\\[3 ex]
   		\AxiomC{$ \Gamma_0 \vdash A $}
  		\AxiomC{$ \Gamma_1 \vdash J(A \rightarrow B)$}
  		\LeftLabel{{\tiny $\rightarrow E$}} 
   		\BinaryInfC{$ \Gamma_0, \Gamma_1 \vdash B$}
   		\DisplayProof
   			&
   		\AxiomC{$ J\Gamma, A \vdash B $}
   		\LeftLabel{{\tiny $\rightarrow I$}} 
   		\UnaryInfC{$ \Gamma \vdash A \rightarrow B$}
   		\DisplayProof
   		\\[3 ex]
   	
	\end{tabular}
\end{center}

\begin{flushleft}
 		\textbf{Additional Rules:}
\end{flushleft}
\begin{center}
  	\begin{tabular}{c c c}
		
		\AxiomC{$ JA \vdash \bot $}
		\LeftLabel{\tiny{$sCoJ$}}
		\UnaryInfC{$ A \vdash \bot $}
		\DisplayProof
		&
		\AxiomC{$ \Gamma \vdash A$}
		\LeftLabel{\tiny{$CoJ$}}
		\UnaryInfC{$ \Gamma \vdash JA$}
		\DisplayProof
		&
		\AxiomC{$ \Gamma \vdash JA$}
		\LeftLabel{\tiny{$J$}}
		\UnaryInfC{$ \Gamma \vdash A$}
		\DisplayProof
		\end{tabular}
\end{center}
Note that in the rules $\rightarrow I$ and $F$, $\Gamma$ can have exactly one element.\\ 

Now define the system minimal $J$-logic as the logic with all the structural and propositional rules. If we add to this logic the rule $sCoJ$ we will have the logic $\mathbf{sCoJ}$ and similarly for the other rules. Note that we show the resulting system of adding both of the rules $sCoJ$ and $J$ as $\mathbf{sI}$ and the system with the rules $J$ and $CoJ$ as $\mathbf{I}$.
\begin{dfn}\label{t4-1}
A topological model for $J$-logics is a tuple $(X, J, V)$ such that $(X, J)$ is a modal space and $V:\mathcal{L}_{J} \to X$ is a valuation function such that: 
\begin{itemize}
\item[$(i)$]
$V(\top)=1$ and $V(\bot)=0$.
\item[$(ii)$]
$V(A \wedge B)=V(A) \wedge V(B)$.
\item[$(iii)$]
$V(A \vee B)=V(A) \vee V(B)$.
\item[$(iv)$]
$V(A \rightarrow B)= V(A) \rightarrow_J V(B)$.
\item[$(v)$]
$V(JA)=JV(A)$.
\end{itemize}
We say $(X, J, V) \vDash \Gamma \Rightarrow A$ when $\bigwedge_{\gamma \in \Gamma} V(\gamma) \leq V(A)$ and $(X, J) \vDash \Gamma \Rightarrow A$ when for all $V$, $(X, J, V) \vDash \Gamma \Rightarrow A$. Moreover, if for some fix $X$ and for all $(X, J)$ in some class $\mathcal{C}$ we have $(X, J) \vDash \Gamma \Rightarrow A$, then we write $X \vDash_{\mathcal{C}} \Gamma \Rightarrow A$. Furthermore,  we omit the symbol $\Rightarrow$ in all the cases where $\Gamma$ is empty. 
\end{dfn}
As we expect, we have the following soundness theorem.
\begin{thm}\label{t4-2}(Soundness Theorem)
\begin{itemize}
\item[$(i)$]
If $\Gamma \vdash_{\mathbf{mJ}} A$ then $\mathbf{MS} \vDash \Gamma \Rightarrow A$.
\item[$(ii)$]
If $\Gamma \vdash_{\mathbf{sCoJ}} A$ then $\mathbf{sCoTS} \vDash \Gamma \Rightarrow A$.
\item[$(iii)$]
If $\Gamma \vdash_{\mathbf{CoJ}} A$ then $ \mathbf{CoTS} \vDash \Gamma \Rightarrow A$.
\item[$(iv)$]
If $\Gamma \vdash_{\mathbf{J}} A$ then $ \mathbf{TS} \vDash \Gamma \Rightarrow A$.
\item[$(v)$]
If $\Gamma \vdash_{\mathbf{sI}} A$ then $ \mathbf{sS} \vDash \Gamma \Rightarrow A$.
\item[$(vi)$]
If $\Gamma \vdash_{\mathbf{I}} A$ then $ \mathbf{S} \vDash \Gamma \Rightarrow A$.
\end{itemize}
\end{thm}
\begin{proof}
Since the logics are just the syntactical representations of the structure of the modal spaces, the soundness theorem is clear and we will skip the details of its proof.
\end{proof}
To prove the completeness, we need the following lemma:
\begin{lem}\label{t4-3}
Let $\mathbb{A}=(A, \leq, \wedge, \vee, 0, 1, J, \rightarrow)$ be a structure where $(A, \leq)$ is a poset with meet $\wedge$, joint $\vee$, the zero element $0$, the one element $1$ and $J$ is monotone such that $J(-)\wedge a$ is a left adjoint to $a \rightarrow (-)$ for all $a \in A$. Then there exists a modal space $(X, I)$ such that the structure $\mathbb{A}$ is embeddable in $(X, I)$. Moreover, if $\mathbb{A}$ has any of the properties of the Definition \ref{t2-8}, so does $(X, I)$.
\end{lem}
\begin{proof}
The proof is just the usual proof of the Stone type duality. Define $X$ as the set of all downward-closed subsets of $A$ which includes $0$ and closed under all joins. Define the order and the meet as the inclusion and the intersection and the join of $\{U_i\}_{i \in I}$ as 
\[
\{z| \exists u_i \in U_i, z \leq \bigvee_{i \in I} u_i\}
\]
Also define $1$ as $A$, $0$ as $\{0\}$ and $I(V)=\{u| u \leq Jv, v \in V\}$. It is not hard to check that $X$ is a locale and $I$ is join-preserving. Hence, $(X, I)$ is a modal space. Now, define $U : A \to X$ as $U(a)=\{b|b \leq a\}$. Since $U(a)$ is downward-closed and closed under joins, $U$ is well-defined. It is easy to prove that $U$ preserves all the lattice structure including $J$ and $\rightarrow$ and all the properties of the Definition \ref{t2-8}.
\end{proof}
\begin{thm}\label{t4-4}(Completeness Theorem)
\begin{itemize}
\item[$(i)$]
If $ \mathbf{MS} \vDash \Gamma \Rightarrow A$ then $\Gamma \vdash_{\mathbf{mJ}} A$.
\item[$(ii)$]
If $ \mathbf{sCoTS} \vDash \Gamma \Rightarrow A$ then $\Gamma \vdash_{\mathbf{sCoJ}} A$.
\item[$(iii)$]
If $ \mathbf{CoTS} \vDash \Gamma \Rightarrow A$ then $\Gamma \vdash_{\mathbf{CoJ}} A$.
\item[$(iv)$]
If $ \mathbf{TS} \vDash \Gamma \Rightarrow A$ then $\Gamma \vdash_{\mathbf{J}} A$.
\item[$(v)$]
If $ \mathbf{sS} \vDash \Gamma \Rightarrow A$ then $\Gamma \vdash_{\mathbf{sI}} A$.
\item[$(vi)$]
If $ \mathbf{S} \vDash \Gamma \Rightarrow A$ then $\Gamma \vdash_{\mathbf{I}} A$.
\end{itemize}
\end{thm}
\begin{proof}
Pick one of the logics that we want to prove the completeness for. By $\vdash$ we mean the provability in that system. The proof is the usual Lindenbaum type proof augmented with the embedding of Lemma \ref{t4-3}. Define $X$ to be the set of all formulas of the language $\mathcal{L}_J$. Define the equivalence relation $\equiv$ as $A \equiv B$ iff $A \vdash B$ and $B \vdash A$. It is clear that $(X/\equiv, \vdash)$ is a poset with all finite meets and joins. Moreover it has the zero and one elements and $J$ and $\rightarrow$ such that $J(-) \times A$ is a left adjoint for $A \rightarrow (-)$. By Lemma \ref{t4-3}, there exists a modal space $(Y, I)$ into which we can embed the structure of $X$ by some morphism $\phi$. Define $V(p)=\phi(p)$. Therefore, it is easy to check that for all formulas $C$, $V(C)=\phi(C)$. Therefore, since $(Y, I, V) \vDash \Gamma \vdash A$ we have $V(\bigwedge \Gamma) \leq V(A)$. Hence $\phi(\bigwedge \Gamma) \leq \phi(A)$. Since $\phi$ is an embedding, $\Gamma \vdash A$ which completes the proof.
\end{proof}
\begin{thm}\label{t4-5}
Assume $X$ and $Y$ are modal spaces, $Y$ is Alexandrov and $f: X \to Y$ is a continuous surjection. Then for any formula $A \in \mathcal{L}_J$, if $X \vDash_{\mathcal{C}} A$ then $Y \vDash_{\mathcal{C}} A$ for any class $\mathcal{C}$ as in the Definition \ref{t2-8}.
\end{thm}
\begin{proof}
It is similar to the proof of the Theorem \ref{t3-8}.
\end{proof}
\begin{thm}\label{t4-6}(Completeness Theorem, Strong version) Let $X$ be a topological space with continuum many connected components, then:
\begin{itemize}
\item[$(i)$]
If $ X \vDash_{\mathbf{MS}} \Gamma \Rightarrow A$ then $\Gamma \vdash_{\mathbf{mJ}} A$.
\item[$(ii)$]
If $ X \vDash_{\mathbf{sCoTS}} \Gamma \Rightarrow A$ then $\Gamma \vdash_{\mathbf{sCoJ}} A$.
\item[$(iii)$]
If $ X \vDash_{\mathbf{CoTS}} \Gamma \Rightarrow A$ then $\Gamma \vdash_{\mathbf{CoJ}} A$.
\item[$(iv)$]
If $ X \vDash_{\mathbf{TS}} \Gamma \Rightarrow A$ then $\Gamma \vdash_{\mathbf{J}} A$.
\item[$(v)$]
If $ X \vDash_{\mathbf{sS}} \Gamma \Rightarrow A$ then $\Gamma \vdash_{\mathbf{sI}} A$.
\item[$(vi)$]
If $ X \vDash_{\mathbf{S}} \Gamma \Rightarrow A$ then $\Gamma \vdash_{\mathbf{I}} A$.
\end{itemize}
\end{thm}
\begin{proof}
The proof is similar to the proof of the Theorem \ref{t3-9}. Note that the spaces that we used for completeness are Alexandrof and has cardinality at most equal to the continuum and any topological space with this size is a continuous image of a topological space with continuum many connected components.
\end{proof}
Now let us review some important sub-intuitionistic logics which were introduced in \cite{Vi}, \cite{Ard} and \cite{Do}. For this purpose, we need to introduce some rules in the usual natural deduction system:

\begin{flushleft}
   \textbf{Propositional Rules:}
   \end{flushleft}
\begin{center}
  	\begin{tabular}{c c}
        
        \AxiomC{$A$}
    		\LeftLabel{\tiny{$\top$}}
    		\UnaryInfC{$\top$}
    		\DisplayProof
    		& \;\;
    		\AxiomC{$\bot$}
    		\LeftLabel{\tiny{$\bot$}}
    		\UnaryInfC{$A$}
    		\DisplayProof
    		\\[4 ex]  	    
  	    
		\AxiomC{$ A$}
	   	\AxiomC{$ B$}
	   	\LeftLabel{\tiny{$\wedge I$}}
	   	\BinaryInfC{$A \wedge B $}
	   	\DisplayProof
		&
		\AxiomC{$ A \wedge B $}
		\LeftLabel{\tiny{$\wedge E$}}
		\UnaryInfC{$ A$}
		\DisplayProof
			
		\AxiomC{$ A \wedge B $}
		\LeftLabel{\tiny{$\wedge E$}}
		\UnaryInfC{$ B$}
		\DisplayProof
	   		\\[4 ex]
	   		
        \AxiomC{$  A $}
        \LeftLabel{\tiny{$\vee I$}}
   		\UnaryInfC{$ A \vee B$}
   		\DisplayProof
   		
   		\AxiomC{$  B $}
   		\LeftLabel{\tiny{$\vee I$}}
   		\UnaryInfC{$ A \vee B$}
   		\DisplayProof
			&
   		\AxiomC{$A \lor B$}
        \AxiomC{[$A$]}
        \noLine
   		\UnaryInfC{$\mathcal{D}$}
        \noLine
        \UnaryInfC{$C$}
        
        \AxiomC{[$B$]}
        \noLine
   		\UnaryInfC{$\mathcal{D'}$}
        \noLine
        \UnaryInfC{$C$}
        \LeftLabel{\tiny{$\vee E$}}
        \TrinaryInfC{$C$}
        \DisplayProof
    		\\[6 ex]
   	
   \end{tabular}
   \end{center}

\begin{center}
\begin{tabular}{c}

   		\AxiomC{[$A$]}
   		\noLine
   		\UnaryInfC{$\mathcal{D}$}
   		\noLine
   		\UnaryInfC{$B$}
   		\LeftLabel{\tiny{$\rightarrow I$}}
   		\UnaryInfC{$ A \rightarrow B$}
   		\DisplayProof 
   		
\end{tabular}
\end{center}

\begin{flushleft}
   \textbf{Formalized Rules:}
   \end{flushleft}
\begin{center}
\begin{tabular}{c c}
   		   		
   		\AxiomC{$A \rightarrow B$}
    		\AxiomC{$A \rightarrow C$}
    		\LeftLabel{\tiny{$(\wedge I)_f$}}
    		\BinaryInfC{$A \rightarrow B \wedge C$}
    		\DisplayProof
    		&
    		\AxiomC{$A \rightarrow C$}
    		\AxiomC{$B \rightarrow C$}
    		\LeftLabel{\tiny{$(\vee E)_f$}}
    		\BinaryInfC{$A \vee B \rightarrow C$}
    		\DisplayProof
    		\\[4 ex]
    		
\end{tabular}
\end{center}
\begin{center}
      \begin{tabular}{c}
  
    		\AxiomC{$A \rightarrow B$}
    		\AxiomC{$B \rightarrow C$}
    		\LeftLabel{\tiny{$tr_f$}}
    		\BinaryInfC{$A \rightarrow C$}
    		\DisplayProof
    		\\[4 ex]
   		
	\end{tabular}
\end{center}
Moreover, consider the following additional rules:
\begin{flushleft}
   \textbf{Additional Rules:}
   \end{flushleft}
\begin{center}
  	\begin{tabular}{c c c}
		\AxiomC{$\top \rightarrow \bot$}
		\LeftLabel{\tiny{$E$}}
		\UnaryInfC{$\bot$}
		\DisplayProof
		&
        \AxiomC{$ A \wedge (A \rightarrow B)$}
        \LeftLabel{\tiny{$T$}}
        \UnaryInfC{$ B $}
        \DisplayProof
        &
        \AxiomC{$A$}
		\LeftLabel{\tiny{$cur$}}
		\UnaryInfC{$\top \to A$}
		\DisplayProof
		\\[3 ex]
	\end{tabular}
\end{center}
The condition for the rule $\rightarrow I$ is that $A$ should be the only assumption to deduce $B$. The logic $\mathbf{KPC}$ is defined as the system which consists of all the propositional rules and all the formalized rules. Then $\mathbf{BPC}$ is defined as $\mathbf{KPC} + Cur$, $\mathbf{EKPC}$ as $\mathbf{KPC}$ plus the rule $E$, $\mathbf{EBPC}$ as $\mathbf{BPC}$ plus the rule $E$, $\mathbf{KTPC}$ as $\mathbf{KPC}$ plus the rule $T$ and $\mathbf{IPC}$ is defined as $\mathbf{BPC} + T$.
\begin{rem}\label{t4-7}
The other way to define $\mathbf{BPC}$ is by using the rules for $\mathbf{KPC}$ with relaxing the condition on $\rightarrow I$ (See \cite{Ard2}). It is clear that the rule $cur$ is provable by this more strong version of $\rightarrow I$. For the converse, first we will show that using the rule $cur$, it is possible to prove that $C \vdash D \rightarrow C$ for all the formulas $C$ and $D$. It is enough to use the $cur$ rule on $C$ to show $C \vdash \top \rightarrow C$ and then note $D \vdash \top$, hence $\vdash D \to \top$. By $tr$ and formalized $tr$ we have $C \vdash D \rightarrow C$.\\
Then assume $\Gamma, A \vdash B$ hence $\bigwedge \Gamma, A \vdash B$ and by the original version of $\rightarrow I$ we have $\vdash \bigwedge \Gamma \wedge A \to B$. But by what we proved and formalized $\wedge$ we can prove $\bigwedge \Gamma \vdash A \to \bigwedge \Gamma \wedge A$, hence by $tr$, we have $\bigwedge \Gamma \vdash A \to B$ which is equivalent to $\Gamma \vdash A \to B$.
\end{rem}
\begin{rem}\label{t4-8}
Note that the language of these sub-intuitionistic logics is a subset of the language of $J$-logics. Therefore, we can apply the topological semantics for the $J$-formulas also for these sub-intuitionistic logics.
\end{rem}
The following example shows how the interpretation of the usual propositional formulas works.
\begin{exam}\label{t4-9}
Let $a$ and $b$ be two real numbers and $a \neq 0$. Now consider the model $X_{a, b}=(\mathbb{R}, \tau_E, J_{a,b})$ where $\tau$ is the Euclidean topology and $J_{a,b} : \tau \to \tau$ is the following function: $J_{a,b}(U)=aU+b=\{ax+b | x \in U \}$. It is clear that the image of any open $U$ is also open. $J$ is also monotone and join preserving. Now, we want to show that $X_{\frac{1}{2}, 0} \nvDash (\top \to p) \Rightarrow p$. Pick $V(p)=(0, 1)$. Then we have $V(\top \rightarrow p)= \bigcup\{U | \frac{1}{2} U \subseteq (0, 1)\}=(0, 2)$. Hence, $V(\top \rightarrow p)=(0, 2) \nsubseteq (0, 1)=V(p)$.\\
Now consider $X_{2, 0}$. We want to show that $X_{2, 0} \nvDash p \Rightarrow \top \rightarrow p$. Pick $V'(p)=(0, 2)$ then $V'(\top \rightarrow p)=(0, 1)$. Hence $V'(p) \nsubseteq V'(\top \to p)$.
\end{exam}

\begin{thm}\label{t4-10}(Embedding) Assume $\Gamma \cup \{A\} \subseteq \mathcal{L}$ where $\mathcal{L}$ is the language of propositional logic, then:
\begin{itemize}
\item[$(i)$]
$\Gamma \vdash_{\mathbf{mJ}} A$ iff $ \Gamma \vdash_{\mathbf{KPC}} A$.
\item[$(ii)$]
If $\Gamma \vdash_{\mathbf{sCoJ}} A$ iff $ \Gamma \vdash_{\mathbf{EKPC}} A$.
\item[$(iii)$]
If $\Gamma \vdash_{\mathbf{CoJ}} A$ iff $ \Gamma \vdash_{\mathbf{KTPC}} A$.
\item[$(iv)$]
If $\Gamma \vdash_{\mathbf{J}} A$ iff $ \Gamma \vdash_{\mathbf{BPC}} A$.
\item[$(v)$]
If $\Gamma \vdash_{\mathbf{sI}} A$ iff $ \Gamma \vdash_{\mathbf{EBPC}} A$.
\item[$(vi)$]
If $\Gamma \vdash_{\mathbf{I}} A$ iff $ \Gamma \vdash_{\mathbf{IPC}} A$.
\end{itemize}
\end{thm}
\begin{proof}
To prove the soundness part, note that all the propositional rules except $\to I$ are available in $\mathbf{mJ}$. Therefore, it remains to show that all the formalized rules and $\rightarrow I$ are also provable in $\mathbf{mJ}$. This is what we will do in the following proof trees. Note that by double line rules we mean that there is an easy omitted proof tree between the upper part and the lower part. The main rule among them is the following tree:\\

\begin{center}
\begin{tabular}{c}
        
        \AxiomC{$A\wedge B \vdash A \wedge B $}
        \LeftLabel{\tiny{$\wedge E$}}
        \UnaryInfC{$A \wedge B \vdash A$}
        \LeftLabel{\tiny{$mJ$}}
        \UnaryInfC{$J(A \wedge B) \vdash J(A)$}

        \AxiomC{$A\wedge B \vdash A \wedge B $}
        \LeftLabel{\tiny{$\wedge E$}}
        \UnaryInfC{$A \wedge B \vdash B$}
        \LeftLabel{\tiny{$mJ$}}
        \UnaryInfC{$J(A \wedge B) \vdash J(B)$}
        \LeftLabel{\tiny{$\wedge I$}}
        \BinaryInfC{$J(A \wedge B) \vdash J(A) \wedge J(B)$}
        \DisplayProof
        
\end{tabular}
\end{center}

For the formalized $\wedge I$, we have:

\begin{center}
\begin{tabular}{c}
      
        \AxiomC{$ $}
        \doubleLine
        \UnaryInfC{$J(A \rightarrow B), A \vdash B$}
        
	    \AxiomC{$ $}
        \doubleLine
        \UnaryInfC{$J(A \rightarrow C), A \vdash C$}
	    \LeftLabel{\tiny{$\wedge I$}}
        \BinaryInfC{$J(A \rightarrow B), J(A \rightarrow C), A \vdash B \wedge C$}
        \doubleLine
	    \UnaryInfC{$J((A \rightarrow B) \wedge (A \rightarrow C)), A \vdash B \wedge C$}
	    \LeftLabel{\tiny{$\rightarrow I$}}
        \UnaryInfC{$(A \rightarrow B) \wedge (A \rightarrow C) \vdash A \rightarrow (B \wedge C)$}
        \doubleLine
        \UnaryInfC{$(A \rightarrow B), (A \rightarrow C) \vdash A \rightarrow (B \wedge C)$}
        \DisplayProof
        
\end{tabular}
\end{center}

and for the formalized $\vee I$, we have:

\begin{center}
\begin{tabular}{c}
      \AxiomC{$ $}
        \doubleLine
        \UnaryInfC{$J(A \rightarrow C), A \vdash C$}
        
	    \AxiomC{$ $}
        \doubleLine
        \UnaryInfC{$J(B \rightarrow C), B \vdash C$}
	    \LeftLabel{\tiny{$\vee I$}}
        \BinaryInfC{$J(A \rightarrow C), J(B \rightarrow C), A \vdash B \vee C$}
        \doubleLine
	    \UnaryInfC{$J((A \rightarrow C) \wedge (B \rightarrow C)), A \vdash B \vee C$}
	    \LeftLabel{\tiny{$\rightarrow I$}}
        \UnaryInfC{$(A \rightarrow C) \wedge (B \rightarrow C) \vdash A \rightarrow B \vee C$}
        \doubleLine
        \UnaryInfC{$(A \rightarrow C), (B \rightarrow C) \vdash A \rightarrow B \vee C$}
        \DisplayProof
        
\end{tabular}
\end{center}

for the formalized $tr$, we have:

\begin{center}
\begin{tabular}{c}
      
        \AxiomC{$ $}
        \doubleLine
        \UnaryInfC{$J(A \rightarrow B), A \vdash B$}
        
	    \AxiomC{$ $}
        \doubleLine
        \UnaryInfC{$J(B \rightarrow C), B \vdash C$}
	    \LeftLabel{\tiny{$cut$}}
        \BinaryInfC{$J(A \rightarrow B), J(B \rightarrow C), A \vdash C$}
        \doubleLine
	    \UnaryInfC{$J((A \rightarrow B) \wedge (B \rightarrow C)), A \vdash C$}
	    \LeftLabel{\tiny{$\rightarrow I$}}
        \UnaryInfC{$(A \rightarrow B) \wedge (B \rightarrow C) \vdash A \rightarrow C$}
        \doubleLine
        \UnaryInfC{$(A \rightarrow B), (B \rightarrow C) \vdash A \rightarrow C$}
        \DisplayProof
\end{tabular}
\end{center} 

And finally for $\rightarrow I$ we have:
\begin{center}
\begin{tabular}{c}
      
        \AxiomC{$\vdash \top $}
        
        \AxiomC{$J(\top), A \vdash A $}
        \AxiomC{$ A \vdash B $}
        \LeftLabel{\tiny{$cut$}}
        \BinaryInfC{$J(\top) , A \vdash B$}
        	\LeftLabel{\tiny{$\rightarrow I$}}
        \UnaryInfC{$\top \vdash A \to B$}
        \LeftLabel{\tiny{$cut$}}
        \BinaryInfC{$ \vdash A \rightarrow B$}
        \DisplayProof
\end{tabular}
\end{center}
Now we have to show that the additional rules are provable by their corresponding rules. For $cur$, we will use its characterization based on $\rightarrow I$ as mentioned in the Remark \ref{t4-7}. 

\begin{center}
\begin{tabular}{c}
      
        \AxiomC{$ J(\bigwedge \Gamma) \vdash J(\bigwedge \Gamma) $}
        \LeftLabel{\tiny{$J$}}
        \UnaryInfC{$J(\bigwedge \Gamma) \vdash \bigwedge \Gamma $}
        \AxiomC{$\Gamma, A \vdash B $}
        \doubleLine
        \UnaryInfC{$\bigwedge \Gamma , A \vdash B$}
        	\LeftLabel{\tiny{$cut$}}
        \BinaryInfC{$J(\bigwedge \Gamma) , A \vdash B$}
        \LeftLabel{\tiny{$\rightarrow I$}}
	    \UnaryInfC{$\bigwedge \Gamma \vdash A \rightarrow B$}
	    \doubleLine
        \UnaryInfC{$ \Gamma \vdash A \rightarrow B$}
        \DisplayProof
\end{tabular}
\end{center}

For $T$ and $E$ we have:

\begin{center}
\begin{tabular}{c c}
      
        \AxiomC{$A \vdash A $}
        
        \AxiomC{$ A \rightarrow B \vdash A \rightarrow B $}
        \LeftLabel{\tiny{$CoJ$}}
        \UnaryInfC{$A \rightarrow B \vdash J(A \rightarrow B) $}
        	\LeftLabel{\tiny{$\rightarrow E$}}
        \BinaryInfC{$A, A \rightarrow B \vdash B$}
        \DisplayProof
        &\;
        
        \AxiomC{$ $}
        \doubleLine
        \UnaryInfC{$\top, J(\top \rightarrow \bot) \vdash \bot $}
        \doubleLine
        \UnaryInfC{$J(\top \rightarrow \bot) \vdash \bot $}
        	\LeftLabel{\tiny{$sCoJ$}}
        \UnaryInfC{$\top \rightarrow \bot \vdash \bot$}
        \DisplayProof
\end{tabular}
\end{center}
To prove the converse, notice that all the Kripke models of the propositional logics can be translated to the corresponding modal spaces. For this purpose, we have to divide the proof into two different cases: The first case is when the rule $cur$ is present in the propositional system and the second case is when it is not. For the second case, assume $(W, R, V)$ is the Kripke model, then define the modal space $(P(W), J_R)$ of this Kripke model as in the Example \ref{t2-3} and define $V'(p)=V(p)$. It is easy to see that for any propositional formula $B$, $w \vDash B$ iff $w \in V'(B)$. The proof is exactly the same as the proof of Theorem \ref{t3-7}. Therefore, since the sequent $\Gamma \vdash A$ is provable in the corresponding $J$-logic, we know $(P(W), J_R, V') \vDash \Gamma \Rightarrow A$ hence $V'(\bigwedge \Gamma) \leq V'(A)$ which implies $(W, R, V) \vDash \Gamma \Rightarrow A$. Finally, using the completeness of the Kripke Semantics for the propositional logic we can prove the provability of $\Gamma \vdash A$ in the propositional logic.\\

The proof of the first case is the same as the previous case with just a little change in the definition of the corresponding modal space. In this case we have to define the modal space as in the Example \ref{t2-4}. The reason is that $R$ is transitive and we want $J_R$ to have the property of being temporal. Now define $V'(p)=V(p)$. Note that by the definition of Kripke models for logics with the $cur$ rule, $V(p)$ is upward-closed which implies that it is open in our topological space. Hence $V'(p)$ is well-defined. The next important part is the claim that for any propositional formula $B$, $w \vDash B$ iff $w \in V'(B)$. The atomic case is trivial by definition. The cases for conjucntion and disjunction are also trivial. For the implication case, assume 
\[
w \in V'(A \rightarrow B)=\bigcup\{U| JU \cap V'(A) \subseteq V'(B)\}.
\]
Therefore, there exists $U$ such that $w \in U$ and $JU \cap V'(A) \subseteq V'(B)$. Hence if $(w, u) \in R$, then $u \in JU$ and thus we know that $u \in V'(A)$ implies $u \in V'(B)$. By IH, if $u \vDash A$ then $u \vDash B$. Hence $w \vDash A \rightarrow B$.
Conversely, if $w \vDash A \rightarrow B$ then for all $u$ such that $(w, u) \in R$ if $u \vDash A$ then $u \vDash B$. By using IH it implies that if $u \in V'(A)$ then $u \in V'(B)$. Therefore define $U=\{u| (w, u) \in R \vee u=w \}$. Since $R$ is transitive, $U$ is upward-closed and hence open. On the other hand, by the definition of $J$ we have $JU=\{u| (w, u) \in R\}$ and since $JU \cap V'(A) \subseteq V'(B)$ we will have $w \in U \subseteq V'(A \rightarrow B)$. Hence $w \in V'(A \rightarrow B)$.\\
The rest of the proof is similar to the second case.
\end{proof}
\begin{rem}\label{t4-11}
The sub-intuitionistic logics that we have defined in this section have an extreme importance in the philosophical sense. However, their limited language and specifically their lack of smooth rules for implication makes developing mathematics on top of them, practically impossible. The reason, roughly speaking, is the anti-symmetry in the definition of the implication. Although almost all of them have a strong introduction rule for implication, almost none of them have a reasonable elimination rule. In the categorical terms, the implication in these systems is not a part of an adjunction. The embedding theorem actually solves exactly this problem. It shows that it is possible to embed these logics to some well-behaved $J$-logics which keep the adjunctive symmetry of the rules of logic while it has some undefined weak modality $J$ to make some room for changes.
\end{rem}
By the embedding theorem, we have the soundness-completeness for the propositional logics with respect to the topological semantics.
\begin{thm}\label{t4-12}(Soundness-Completeness Theorem)
\begin{itemize}
\item[$(i)$]
$\Gamma \vdash_{\mathbf{KPC}} A$ iff $ \Gamma \vDash_{\mathbf{MS}} A$.
\item[$(ii)$]
If $\Gamma \vdash_{\mathbf{EKPC}} A$ iff $ \Gamma \vDash_{\mathbf{sCoTS}} A$.
\item[$(iii)$]
If $\Gamma \vdash_{\mathbf{KTPC}} A$ iff $ \Gamma \vDash_{\mathbf{CoTS}} A$.
\item[$(iv)$]
If $\Gamma \vdash_{\mathbf{BPC}} A$ iff $ \Gamma \vDash_{\mathbf{TS}} A$.
\item[$(v)$]
If $\Gamma \vdash_{\mathbf{EBPC}} A$ iff $ \Gamma \vDash_{\mathbf{sS}} A$.
\item[$(vi)$]
If $\Gamma \vdash_{\mathbf{IPC}} A$ iff $ \Gamma \vDash_{\mathbf{S}} A$.
\end{itemize}
\end{thm}
With the same line of reasoning as in the Theorem \ref{t3-9} and \ref{t4-6}, it is clear that we have a stronger version of completeness just by using the topological spaces with infinitely many connected components. However, in the presence of the rule $cur$, the situation becomes more interesting. To explain how, we need the following topological lemma. 
\begin{lem}\label{t4-13}
Let $X$ be an infinite Hausdorff space. Then every finite tree is a sutjective continuous image of $X$. 
\end{lem}
\begin{proof}
Let us first prove the following claim:\\

\textbf{Claim 1.} For any natural numbers $N$ and $K$, there exists a natural number $M$ such that for any Hausdorff space $X$ with cardinality greater than or equal to $M$, there are $K$ many open subspaces of $X$ with at least $N$ elements.\\

We prove the claim by induction on $N$. For $N=1$, pick $M=K$ and prove by induction. For $K=1$, it is enough to pick the whole space as the open subset. To prove the claim for $K+1$, since we have at least $K+1$ elements, we have also at least $K$ elements, and by IH, it is possible to find at least $K$ non-empty open subsets  $\{U_i\}_{i=0}^K$. Pick $\{x_i\}_{i=0}^K$ as elements such that $x_i \in U_i$. Therefore, there should be some $x \notin \{x_i\}_{i=0}^K$. Now, use the condition that the space is Hausdorff to find $U_{K+1}$ such that $x \in U_{K+1}$ and $U_{K+1}$ is disjoint with all $U_i$'s.\\

Now, if we have the claim for $N$, we want to prove it for $N+1$. By IH we know that there exists $M'$ that works for $N$ and $K'=2K$. We claim that $M=M'$ works for $N+1$ and $K$. If $X$ has at least $M'$ elements, then there are at lest $2K$ mutually disjoint opens such that each of them has at least $N$ elements. If we arrange these $2K$, to $K$ pairs that compute their unions, then we have $K$ opens, each of them contains at least $2N$ elements, which is greater than or equal to $N$.\\

Now we want to prove the following claim:\\

\textbf{Claim 2.} For any natural number $N$, there exists a natural number $M$ such that for any Hausdorff space with at least $M$ elements and any finite tree with at most $N$ elements, there exists a continuous surjection from the space to the tree.\\

We will prove the claim by induction on $N$. For $N=1$ pick $M=1$ and use the constant function. For $N+1$, by IH, we know that for $N$ there exists an $M'$. Pick $M$ as a number in claim 1, for $N$ and $K=M'$. Therefore, the space $X$ has at least $N$ opens each of them contains at least $M'$ elements. Call them $\{U_i\}_{i=1}^N$. Since the tree has $N+1$ element, there are at most $N$ branches for the root such that each of them has at most $N$ nodes. Call these branches $\{T_j\}_{j=0}^n$ for some $n \leq N$. By IH, we can find a surjective continuous function $f_i: U_i \to T_i$ for any $1 \leq i \leq n$. Now define $f: X \to T$ as the extension of the union of $f_i$'s such that it sends any $x \in \bigcup_{i=1}^n U_i$ to the root $r$. The function is clearly surjective. For continuity, note that any open subset of the tree is a upward-closed subset which means that it is equal to $T$ or it is a subset of one of $T_i$'s. For the first case, $f_{-1}(T)=X$ which is open. For the second case, it is implied from the continuity of $f_i$ and the condition that $U_i$ is open.\\

Now, by the claim 2 it is easy to prove the lemma. Let $X$ be an infinite Hausdorff space and $T$ a finite tree. Then for the cardinality of $T$, say $N$, there exists a number $M$ such that for any space with at least $M$ elements, specially $X$, there exists the continuous surjection to $T$.
\end{proof}
\begin{thm}\label{t4-14}(Completeness Theorem, Strong version) Let $X$ be an infinite Hausdorff space. Then
\begin{itemize}
\item[$(i)$]
If $ X \vDash_{\mathbf{TS}} \Gamma \Rightarrow A$ then $\Gamma \vdash_{\mathbf{BPC}} A$
\item[$(ii)$]
If $ X \vDash_{\mathbf{sS}} \Gamma \Rightarrow A$ then $\Gamma \vdash_{\mathbf{EBPC}} A$
\item[$(iii)$]
If $ X \vDash_{\mathbf{S}} \Gamma \Rightarrow A$ then $\Gamma \vdash_{\mathbf{IPC}} A$
\end{itemize}
\end{thm}
\begin{proof}
The claim is a trivial combination of the following three facts: Firstly, the modal spaces constructed from finite Kripke rooted trees (reflexive or serial in the appropriate cases) are complete for the logics. Secondly, these modal spaces are Alexandrov and finally they are surjective continuous image of any infinite Hausdorff space by the Lemma \ref{t4-13}.   
\end{proof}
\section{Categorical Semantics}
In this section we will use different types of categories defined in the first section as natural models for the sub-intuitionistic logics defined in the previous section. The following soundness-completeness result actually shows that the seemingly strange behavior of implication in these logics is actually quite natural.
\begin{dfn}\label{t5-1}
Let $\mathcal{C}$ be a strong category which has product, coproduct, the initial object and the terminal object. We say $\mathcal{C}$ has an internal propositional structure if it internalizes all of its propositional structures, i.e. product, coproduct, the terminal and the initial objects. For instance, for product we should have the following condition: For every objects $A$, $B$ and $C$ there exists $pair : [A, B] \times [A, C] \to [A, B \times C]$ such that the following diagrams commute:

\begin{center}
\begin{tabular}{c c}

\begin{tikzpicture}
\node (C) {$[C, A] \times [C, B]$};
  \node (P) [below of=C] {$[C, A \times B]$};
  \node (Ai) [right of=P, node distance=4cm] {$[C, A]$};
  \draw[->] (C) to node [swap] {$pair$} (P);
  \draw[->] (P) to node [swap] {$[id, p_0]$} (Ai);
  \draw[->] (C) to node {$p_0$} (Ai);
\end{tikzpicture}
&
\begin{tikzpicture}
\node (C) {$[C, A] \times [C, B]$};
  \node (P) [below of=C] {$[C, A \times B]$};
  \node (Ai) [right of=P, node distance=4cm] {$[C, B]$};
  \draw[->] (C) to node [swap] {$pair$} (P);
  \draw[->] (P) to node [swap] {$[id, p_1]$} (Ai);
  \draw[->] (C) to node {$p_1$} (Ai);
\end{tikzpicture}
\end{tabular}
\end{center}
And we have to have the same for all the other parts of the propositional structure.
\end{dfn}
\begin{rem}\label{t5-2}
Note that in the above definition, the natural internal version of the uniqueness condition for the pairs is the condition that the morphism $pair$ should be a mono. However, this condition is redundant in our definition because, using the external uniqueness, it is very easy to prove that the $pair$ is already a mono.
\end{rem}
\begin{dfn}\label{t5-3}
Let $\mathcal{C}$ be a category closed under finite products, finite coproducts and has terminal and initial objects. Also Let $h : \mathcal{C}^{op} \times \mathcal{C} \to \mathcal{C}$ be a functor and $V$ be a function which assigns to any atomic formula in the language $\mathcal{L}$, an object in the category $\mathcal{C}$. Extend $V : \mathcal{L} \to Ob(\mathcal{C})$ as the following:
\begin{itemize}
\item[$(i)$]
$V(\top)=1$ and $V(\bot)=0$.
\item[$(i)$]
$V(A \wedge B)=V(A) \times V(B)$.
\item[$(ii)$]
$V(A \vee B)=V(A) + V(B)$.
\item[$(iii)$]
$V(A \rightarrow B)=h(V(A), V(B))$.
\end{itemize}
We say $(\mathcal{C}, V) \vDash \Gamma \Rightarrow A$, when there exists a morphism $f : \prod_{B \in \Gamma} V(B) \to V(A)$. If for all $V$, $(\mathcal{C}, V) \vDash \Gamma \Rightarrow A$ holds, we say $\mathcal{C} \vDash \Gamma \Rightarrow A$.
\end{dfn}
We have the following soundness and completeness theorems.
\begin{thm}\label{t5-4}(Soundness Theorem)
\begin{itemize}
\item[$(i)$]
If $\Gamma \vdash_{\mathbf{KPC}} A$ then for any category $\mathcal{C}$ with internal propositional structure, $\mathcal{C} \vDash \Gamma \Rightarrow A$.
\item[$(ii)$]
If $\Gamma \vdash_{\mathbf{BPC}} A$ then for any Curry category $\mathcal{C}$ with internal propositional structure, $\mathcal{C} \vDash \Gamma \Rightarrow A$.
\item[$(ii)$]
If $\Gamma \vdash_{\mathbf{IPC}} A$ then for any closed category $\mathcal{C}$ with internal propositional structure, $\mathcal{C} \vDash \Gamma \Rightarrow A$.
\end{itemize}
\end{thm}
\begin{proof}
The proof is by induction on the length of the derivation of $\Gamma \vdash A$. The propositional rules are just easy consequences of the fact that $\times$ and $+$ are product and coproduct. The same is true for $\bot$ and $\top$ as the initial and the terminal objects. The important part is the part of the formalized rules. The formalized rules $(\wedge I)_f$ and $(\vee E)_f$ are satisfied by the internal structure of the category. The $tr_f$ rule is implied by the internal structure of hom. And finally the $\rightarrow I$ rule is satisfied by the natural transformation $\gamma:  \mathcal{C}(A, B) \to \mathcal{C}(1, [A, B])$ which sends $f$ to $[1, f](j_A)$ and internalizes all morphisms from $A$ to $B$. For the additional rules, note that for the $cur$ rule, it suffices to have a natural transformation $i_X: X \to [1, X]$ and for the rule $T$, the converse of $i_X$ works.
\end{proof}

\begin{thm}\label{t5-5}(Completeness Theorem)
\begin{itemize}
\item[$(i)$]
If for any weakly closed category $\mathcal{C}$ with internal propositional structure, $\mathcal{C} \vDash \Gamma \Rightarrow A$, then $\Gamma \vdash_{\mathbf{KPC}} A$.
\item[$(ii)$]
If for any Curry weakly closed category $\mathcal{C}$ with internal propositional structure, $\mathcal{C} \vDash \Gamma \Rightarrow A$, then $\Gamma \vdash_{\mathbf{BPC}} A$.
\item[$(ii)$]
If for any closed category $\mathcal{C}$ with internal propositional structure, $\mathcal{C} \vDash \Gamma \Rightarrow A$, then $\Gamma \vdash_{\mathbf{IPC}} A$.
\end{itemize}
\end{thm}
\begin{proof}
Notice that any modal space with its arrow is a strong category and it is weakly-closed if $J(1)=1$. On the other hand, since the modal spaces constructed from Kripke models are enough to prove completeness, and since in these models we have $J(1)=1$ if the relation $R$ has the property that before any node, there should be another node, it is enough to prove that these logics are complete with respect to their Kripke models with this additional condition. This is obviously the case, because for any corresponding Kripke frame, it is enough to add one reflexive node under all the nodes of the Kripke model. This new model is a model with the condition and furthermore, truth in all the nodes of this model implies the truth in all the nodes of the first model. 
\end{proof}

\section{Modal Topoi and Modal Lambda Calculus}
In the previous sections, we developed the concept of a modal space and its canonical logic. In this section we want to extend these investigations to the higher and more structured level of generalized topological spaces as topoi and generalized logics as modal simply typed lambda calculus. To begin, let us define the notion of weakly Cartesian closed category as the weak version of the usual Cartesian closed categories.
\begin{dfn}\label{t6-1}
A pair $(\mathcal{C}, J)$ is called a weakly Cartesian closed category, wcc, if $\mathcal{C}$ is a category with all finite limits and $J: \mathcal{C} \to \mathcal{C}$ is a functor such that there exists a functor $\rightarrow : \mathcal{C}^{op} \times \mathcal{C} \to \mathcal{C}$ in a way that
\[
\Hom(JC \times A, B) \simeq \Hom(C, A \rightarrow B)
\]
naturally in $A$, $B$ and $C$.\\
Moreover, if there exists a natural transformation $\pi_A: JA \to A$, the category is called temporal and if there exists a natural transformation $\sigma_A \to JA $, it is called cotemporal.
\end{dfn}
Just like the situation in Cartesian closed categories, it is possible to define weak exponential objects by universal morphisms:
\begin{thm}\label{t6-2}
A pair $(\mathcal{C}, J)$ is a weakly Cartesian closed category if it has all finite limits and satisfies the following property: For any two objects $A, B$ there exist an objects $A \rightarrow B$ and a morphism $ev: J(A \rightarrow B) \times A \to B$ such that for any object $C$ and any morphism $f: JC \times A \to B$ there is a unique morphism $\lambda f: C \to (A \to B)$ such that $ ev(J \lambda f \times id)=f$.
\end{thm}
\begin{proof}
It is similar to the usual proof for the usual exponential objects.
\end{proof}
\begin{exam}\label{t6-3}
All modal spaces are weakly Cartesian closed. All Cartesian closed categories are also wcc.
\end{exam}

\begin{thm}\label{t6-4}
Let $(\mathcal{C}, J)$ be a weakly Cartesian closed category, then
\begin{itemize}
\item[$(i)$]
$\mathcal{C}$ equipped with $\rightarrow_J$ is a strong category.
\item[$(ii)$]
$\mathcal{C}$ equipped with $\rightarrow_J$ is weakly closed iff $J(1)\simeq 1$.
\item[$(iii)$]
$\mathcal{C}$ equipped with $\rightarrow_J$ is Curry category iff there exists a natural transformation $\pi : A \to [1, A]$.
\end{itemize}
\end{thm}
\begin{proof}
For $(i)$, define $j_A: 1 \to (A \rightarrow_J A)$ as $\lambda p_1$ and 
\[
L: ((A \rightarrow_J B)\times (B \to_J C)) \to (A \to_J C)
\] 
as $\lambda ev_{B, C}(Jp_1, ev_{A, B}(Jp_0 \times id))$. It is not hard to prove that all the conditions of the Definition \ref{t1-1}, hold. $(ii)$ and $(iii)$ are similar.
\end{proof}
Now it is time to lift the idea of the beginning of the third section from the level of propositions and truth values to the level of types and constructions. To do so, it is natural to replace the notion of space by its generalized version of Grothendieck topoi and then adding a functor $J$ to import the needed part of the notion of time to implement the lifted version of predicative implications which are function spaces now. It is also possible to interpret the functor $J$ as the higher order version of the Lawvere-Tierny topology which lifts that morphism from the level of subobjects and inclusion to the level of objects and morphisms. 
\begin{dfn}\label{t6-5}
A pair $(\mathcal{E}, J)$ is called a modal topos if $\mathcal{E}$ is a Grothendieck topos and $J: \mathcal{E} \to \mathcal{E}$ is a colimit preserving functor. Moreover, if there exists a natural transformation $\pi_A: JA \to A$, the modal topos is called temporal and if there exists a natural transformation $\sigma_A \to JA $, it is called cotemporal.
\end{dfn}
\begin{exam}\label{t6-6}
Assume that $\mathcal{E}$ is a Grothendieck topos and $f^* \dashv f_*: \mathcal{E} \to \mathcal{E}$ is a geometric morphism, then $(\mathcal{E}, f^*)$ is a modal topos. The only thing that we have to check is the colimit preserving condition of $f^*$ which is evident by the fact that it is a left adjoint.  
\end{exam}
\begin{exam}\label{t6-7}
Assume that $\mathbb{C}$ is a small category and $J: \mathbb{C} \to \mathbb{C}$ is a functor, then $(\Set^{\mathbb{C}^{op}}, J^*)$ where $J^*(F)=F \circ J $ is a modal topos. As a concrete example of such a modal topos, assume that $\mathbb{C}=(\mathbb{N}, \leq)$ and $J: \mathbb{N} \to \mathbb{N}$ as $J(n)=J(n\dot{-}1)$. Then the modal topos of presheaves over $\mathbb{N}$ is actually the space of variable sets (constructions) on which $J^*$ acts as a pulling back operator to pull back a variable construction one level on the line of time. 
\end{exam}
Just like modal spaces, modal topoi have a natural weak exponential object or as it seems natural to say, a predicative function space.
\begin{thm}\label{t6-8}
Let $(\mathcal{E}, J)$ be a modal topos. Then there exists a functor $[-, -]: \mathcal{E}^{op} \times \mathcal{E} \to \mathcal{E}$ such that 
\[
\Hom(JC \times A, B) \simeq \Hom(C, [A, B])
\]
naturally in $A$, $B$ and $C$. Hence, any modal topos is a weakly Cartesian closed category.
\end{thm}
\begin{proof}
Since $\mathcal{E}$ is a Grothendieck topos, and $J: \mathcal{E} \to \mathcal{E}$ is a colimit preserving functor, by the adjoint functor theorem, it has a right adjoint $I: \mathcal{E} \to \mathcal{E}$. Now, define $[A,B]=I(B^A)$. This funcor has the property because
\[
\Hom(JC \times A, B) \simeq \Hom(JC, B^A) \simeq \Hom(C, I(B^A))=\Hom(C, [A, B]).
\]
\end{proof}
\begin{thm}\label{t6-9}
Let $(\mathcal{F}, J)$ be a modal topos, $\mathcal{E}$ be a Grothendieck topos and $f=(f_* \dashv f^*): \mathcal{E} \to \mathcal{F}$ be a connected geometric morphism. Then if $f^*$ has a left adjoint, then there exists a functor $I: \mathcal{E}^{op} \times \mathcal{E} \to \mathcal{E}$ such that $(\mathcal{E}, I)$ is a modal topos and $f^*(A \to_J B)\simeq f^*(A) \to_I f^*(B)$ natural in $A$ and $B$.
\end{thm}
\begin{proof}
Define $I=f^*Jf_!$. First of all, it is clear that 
\[
\Hom(C, f^*(A) \to_I f^*(B)) \simeq \Hom(IC \times f^*(A), f^*(B)).
\]
Then, we have
\[
\Hom(IC \times f^*(A), f^*(B)) \simeq \Hom(f^*Jf_!C \times f^*(A), f^*(B)).
\] 
Since $f^*$ is a right adjoint, it preserves limit, hence 
\[
\Hom(IC \times f^*(A), f^*(B)) \simeq \Hom(f^*(Jf_!C \times A), f^*(B)).
\]
But $f$ is connected, which means that $f^*$ is full and faithful, hence
\[
\Hom(IC \times f^*(A), f^*(B)) \simeq \Hom(Jf_!C \times A), B),
\]
and then
\[
\Hom(IC \times f^*(A), f^*(B)) \simeq \Hom(f_!C, A \rightarrow_J B).
\]
Since $f_! \dashv f^*$, we have
\[
\Hom(IC \times f^*(A), f^*(B)) \simeq \Hom(C, f^*(A \rightarrow_J B)).
\]
Therefore,
\[
\Hom(C, f^*(A) \to_I f^*(B)) \simeq \Hom(C, f^*(A \rightarrow_J B)).
\]
\end{proof}

Now, let us define the right higher order language to describe weakly Cartesian closed categories and consequently modal topoi to some extent. This language, is an appropriate modal version of simply typed lambda calculus. \\

By modal lambda calculus we mean the following system: The type constructors are $\times$, $+$, $0$, $1$, $J$, $\rightarrow$. The terms constructors are $\langle \cdot, \cdot \rangle$, $p_0$, $p_1$, $r$, $l$, $d$, $*$, $!$, $j$, $\lambda$, $ap$.
We now begin to build up a system of rules.\\

\begin{flushleft}
 		\textbf{Structural Rules:}
\end{flushleft}
\begin{center}
  	\begin{tabular}{c c}
		\AxiomC{$x: A \vdash t(x): B$}
		\LeftLabel{\tiny{$F$}}
		\UnaryInfC{$y: JA \vdash [jt(x)](y): JB$}
		\DisplayProof
		&
		\AxiomC{$\vec{x}: \Gamma \vdash t(\vec{x}): A$}
		\AxiomC{$\vec{z}: \Delta, y:A \vdash s(\vec{z}, y): B$}
		\LeftLabel{\tiny{$cut$}}
		\BinaryInfC{$\vec{x}: \Gamma, \vec{z}: \Delta \vdash s(\vec{z}, t(\vec{x})): B$}
		\DisplayProof
		
		\end{tabular}
\end{center}

\begin{flushleft}
  		\textbf{Propositional Rules:}
\end{flushleft}
\begin{center}
  	\begin{tabular}{c c}
  	
  	    \AxiomC{$\vec{x}: \Gamma \vdash t(\vec{x}): \bot$}
		\LeftLabel{\tiny{$\bot$}}
		\UnaryInfC{$\vec{x}: \Gamma \vdash !t(\vec{x}): A$}
		\DisplayProof
		&
		\AxiomC{$ $}
		\LeftLabel{\tiny{$\top$}}
		\UnaryInfC{$\vec{x}: \Gamma \vdash *: \top$}
		\DisplayProof
       	\\[3 ex]
   		\AxiomC{$\vec{x}: \Gamma \vdash t(\vec{x}): A_0 \times A_1  $}
   		\RightLabel{{\tiny $ (i=0, 1) $}} 
   		\LeftLabel{{\tiny $\times E$}}  		
   		\UnaryInfC{$\vec{x}: \Gamma \vdash p_i(t(\vec{x})): A_i $}
   		\DisplayProof
	   		&
   		\AxiomC{$\vec{x}: \Gamma  \vdash t(\vec{x}): A$}
   		\AxiomC{$\vec{y}: \Delta  \vdash s(\vec{y}): B$}
   		\LeftLabel{{\tiny $\times I$}} 
   		\BinaryInfC{$\vec{x}: \Gamma, \vec{y}: \Delta \vdash \langle t(\vec{x}), s(\vec{y}) \rangle: A \times B $}
   		\DisplayProof
   		\\[3 ex]
   		\AxiomC{$\vec{x}: \Gamma \vdash t(\vec{x}): A$}
   		\LeftLabel{{\tiny $+ I$}} 
   		\UnaryInfC{$\vec{x}: \Gamma \vdash l(t(\vec{x})): A + B$}
   		\DisplayProof
   		&
   		\AxiomC{$\vec{x}: \Gamma \vdash t(\vec{x}): B$}
   		\LeftLabel{{\tiny $+ I$}} 
   		\UnaryInfC{$\vec{x}: \Gamma \vdash r(t(\vec{x})): A + B$}
   		\DisplayProof
	   		
	\end{tabular}
\end{center}

\begin{center}
  	\begin{tabular}{c}
		\AxiomC{$\vec{x}: \Gamma, a: A  \vdash t(\vec{x}, a): C $}
  		\AxiomC{$\vec{y}: \Delta, b: B  \vdash s(\vec{y}, b): C$}
  		\LeftLabel{{\tiny $+ E$}} 
  		\BinaryInfC{$ \vec{x}: \Gamma, \vec{y}: \Delta, e: A + B \vdash d(a,b; t(\vec{x},a), s(\vec{y}, b), e): C$}
  		\DisplayProof
		\end{tabular}
\end{center}

\begin{center}
  	\begin{tabular}{c}
   		\AxiomC{$x: JC, a: A \vdash t(x, a): B $}
   		\LeftLabel{{\tiny $\rightarrow I$}} 
   		\UnaryInfC{$y: C \vdash [\lambda a.t(x, a)](y) :A \rightarrow B$}
   		\DisplayProof
   		\\[3 ex]
		\end{tabular}
\end{center}

\begin{center}
  	\begin{tabular}{c}
		\AxiomC{$\vec{x}: \Gamma \vdash t(\vec{x}): A $}
  		\AxiomC{$\vec{y}: \Delta \vdash s(\vec{y}): J(A \rightarrow B)$}
  		\LeftLabel{{\tiny $\rightarrow E$}} 
   		\BinaryInfC{$\vec{x}: \Gamma, \vec{y}: \Delta \vdash ap(s(\vec{y}), t(\vec{x})): B$}
   		\DisplayProof
		\end{tabular}
\end{center}

\begin{flushleft}
 		\textbf{Additional Rules:}
\end{flushleft}
\begin{center}
  	\begin{tabular}{c c}
	
		\AxiomC{$\vec{x}: \Gamma \vdash t(\vec{x}): JA$}
		\LeftLabel{\tiny{$J$}}
		\UnaryInfC{$\vec{x}: \Gamma \vdash \pi(t(\vec{x})): A$}
		\DisplayProof
		&
		\AxiomC{$\vec{x}: \Gamma \vdash t(\vec{x}) :A$}
		\LeftLabel{\tiny{$CoJ$}}
		\UnaryInfC{$\vec{x}: \Gamma \vdash \sigma(t(\vec{x})): JA$}
		\DisplayProof
		\end{tabular}
\end{center}
Note that by substitution we mean the usual natural recursive definition. However, it is important to point out that in our system and in terms $[Jt(x)](y)$ and $[\lambda z. t(z, x)](y)$, the variable $x$ is not considered free and the only free variable in these terms is $y$. Therefore, we can just substitute $y$ by some term $s$, and the result of the substitution is $[Jt(x)](s)$ and $[\lambda z. t(z, x)](s)$, respectively.\\
 
Finally, the equality rules ($\beta$ and $\eta$ rules) consist of reflexivity, symmetry, transitivity, being closed under substitution and also all the following rules:\\
\\
For $\top$ and $\bot$ we have:\\
\\
$t=*$, for $t: \top$.\\
$t(x)=!x$ for $ t(x:\bot): A$.\\
\\
For $\times$ we have\\
\\
$p_0(\langle x, y \rangle)=x$ and $p_1(\langle x, y \rangle)=y$.\\
$\langle p_0(x), p_1(x) \rangle=x$.\\
\\
For $+$ we have:\\
\\
$d(a, b;t(\vec{x}, a), s(\vec{y}, b), l(c))=t(\vec{x}, c) $.\\
$d(a, b;t(\vec{x}, a), s(\vec{y}, b), r(c))=s(\vec{y}, c) $.\\
$d(a, b;t(\vec{x}, l(a)), t(\vec{x}, r(b)), e)=t(\vec{x}, e) $.\\
\\
For $\rightarrow$ we have:\\
\\
$ap([j[\lambda y. t(x, y)]](x), y)=t(x, y) $.\\
$[\lambda y.ap([jt(x)], y)](x)=t(x) $.\\
\\
For $J$ we have:\\
\\
$[jt]([js](z))=[jt(s)](z) $.\\
$[jx](y)=y$.\\
\\
And in the presence of $\pi$ or $\sigma$:\\
\\
$t(\pi(s))=\pi([Jt](s)) $.\\
$\sigma(t(s))=[jt](\sigma(s)) $.\\

The system of all the rules is called modal lambda calculus and is denoted by $ mJ\lambda$. In the presence of $J$ or $CoJ$ we denote the system by $J \lambda$ and $CoJ\lambda$ and we call them temporal and cotemporal, respectively. In the presence of both, denote it by $I\lambda$.
\begin{thm}\label{t6-10}(Soundness)
\begin{itemize}
\item[$(i)$]
$mJ\lambda$ is interpretable in any weakly Cartesian closed category. Specifically, it is interpretable in any modal topos.
\item[$(ii)$]
$J\lambda$ is interpretable in any temporal weakly Cartesian closed category. Specifically, it is interpretable in any temporal topos.
\item[$(iii)$]
$CoJ\lambda$ is interpretable in any cotemporal weakly Cartesian closed category. Specifically, it is interpretable in any cotemporal topos.
\end{itemize}
\end{thm}
\begin{proof}
Interpret any type as an object and any term $t(\vec{x}):A$ where $\vec{x} : \Gamma$ as a morphism from $\prod \Gamma$ to $A$. The rest is clear. 
\end{proof}
\begin{lem}\label{t6-11}(Embedding) 
Let $\mathbb{C}$ be a small category and $J : \mathbb{C} \to \mathbb{C}$ a functor. Then there is a modal topos $(\mathcal{E}, I)$ such that $(\mathbb{C}, J)$ is embeddable in the modal topos $(\mathcal{E}, I)$, i.e., there exists an embedding $e: \mathbb{C} \to \mathcal{E}$ in a way that $e$ preserves limit and $I(e(c))=e(J(c))$ and consequently, $e(c) \rightarrow_{I} e(d) \simeq e(c \to_J d)$.
\end{lem}
\begin{proof}
Consider $\bar{J}: \mathbb{C} \to \Set^{\mathbb{C}^{op}}$ as the combination of $J$ and Yoneda embedding, i.e., $\bar{J}=y \circ J$. It is possible to lift $\bar{J}$ to the topos $\Set^{\mathbb{C}^{op}}$, i.e., there exists a colimit preserving $\hat{J}: \Set^{\mathbb{C}^{op}} \to \Set^{\mathbb{C}^{op}}$ such that $\hat{J} \circ y=\bar{J}$, hence $\hat{J}(yc)=\bar{J}(c)=y(Jc)$.
The second part is clear from the first part.
\end{proof}
\begin{thm}\label{t6-12}(Completeness)
The syntax of the modal lambda calculus forms a syntactical weakly Cartesian closed category. Therefore, $mJ\lambda$ is complete with respect to all weakly Cartesian closed categories. Moreover, the modal lambda calculus without coproducts and the zero element forms a syntactical modal topos and hence modal topoi are complete for modal lambda calculus without coproducts and the zero element. The same is true in the presence of temporal and cotemporal conditions on both sides.
\end{thm}
\begin{proof}
Define $\mathcal{C}(T)$ as the syntactic category of the type theory $T$. It is enough to interpret types as objects and terms as morphisms and the other constructors as their canonical interpretation in the categorical terms. It is clear that $\mathcal{C}(T)$ is a weakly Cartesian closed category presumably with some conditions regarding being temporal or cotemporal. For the second part, since the syntactic category is a small category, we can embed $\mathcal{C}(T)$ in a modal topos $(\Set^{\mathcal{C}(T)^{op}}, \hat{J})$ as in the Lemma \ref{t6-11}. Since $J$ and $\hat{J}$ and their exponentials act similarly relative to the Yoneda embedding, we can conclude that the equality of interpreted morphisms in the modal topos implies the equality of terms in the syntactic category $\mathcal{C}(T)$. For the temporal and cotemporal cases, it is sufficient to pick $\hat{\pi}$ and $\hat{\sigma}$ as $y\pi$ and $y\sigma$.
\end{proof}

\vspace{4pt}
\textbf{Acknowledgment.} 
We would like to thank Mark van Atten for his generosity to send the draft of his paper on the analysis of the impredicativity of intuitionistic implication. Also we are thankful to Raheleh Jalali and Masoud Memarzadeh for the invaluable discussions.

\end{document}